\documentclass[12pt,reqno]{amsart}
\usepackage{amsmath,amssymb,verbatim,comment,color}
\usepackage[pagebackref,pdftex,hyperindex]{hyperref}

\usepackage[margin=3cm]{geometry}
\usepackage{color}

\newcommand{\C}{\mathbb{C}}

\newcommand{\R}{\mathbb{R}}
\renewcommand{\S}{\mathbb{S}}

\newcommand{\Z}{\mathbb{Z}}

\newcommand{\ft}{\mathfrak{t}}

\newcommand{\cE}{\mathcal{E}}

\newcommand{\cH}{\mathcal{H}}

\newcommand{\cJ}{\mathcal{J}}

\newcommand{\cL}{\mathcal{L}}
\newcommand{\cM}{\mathcal{M}}

\newcommand{\cO}{\mathcal{O}}

\newcommand{\cT}{\mathcal{T}}

\newcommand{\cX}{\mathcal{X}}

\renewcommand{\a}{\alpha}
\renewcommand{\b}{\beta}
\newcommand{\g}{\gamma}
\renewcommand{\d}{\delta}
\newcommand{\e}{\varepsilon}
\newcommand{\s}{\sigma}

\renewcommand{\r}{\rho}

\renewcommand{\i}{\sqrt{-1}}
\newcommand{\p}{\partial}
\newcommand{\bp}{\bar{\partial}}
\newcommand{\dd}{\frac{\i}{2 \pi} \partial \bar{\partial}}
\newcommand{\ddt}{\frac{d}{dt}}

\newcommand{\cf}{{\rm cf.\ }} 
\newcommand{\eg}{{\rm e.g.\ }} 
\newcommand{\ie}{{\rm i.e.\ }} 

\renewcommand{\Re}{\mathrm{Re}}
\renewcommand{\Im}{\mathrm{Im}}

\DeclareMathOperator{\Aut}{Aut}

\DeclareMathOperator{\diag}{diag}

\DeclareMathOperator{\Hom}{Hom}

\DeclareMathOperator{\Lie}{Lie}
\DeclareMathOperator{\loc}{loc}

\DeclareMathOperator{\Mat}{Mat}

\DeclareMathOperator{\Ric}{Ric}
\DeclareMathOperator{\Rm}{Rm}

\DeclareMathOperator{\Tr}{Tr}
\DeclareMathOperator{\Vol}{Vol}

\renewcommand{\leq}{\leqslant}
\renewcommand{\geq}{\geqslant}

\renewcommand{\hat}{\widehat}
\renewcommand{\tilde}{\widetilde}

\numberwithin{equation}{section}       % Number formulas within sections

\theoremstyle{definition}
{
\newtheorem{prop} {Proposition} [section]
\newtheorem{thm}[prop] {Theorem} 
\newtheorem{dfn}[prop] {Definition}
\newtheorem{lem}[prop] {Lemma}
\newtheorem{conj}[prop] {Conjecture}

}
\theoremstyle{remark}
\newtheorem*{ackn}{\bf{Acknowledgment}} 
\newtheorem{rk}[prop]{Remark}

\title{On the modified $J$-equation} 
\date{\today}

\author[R. Takahashi]{Ryosuke Takahashi}
\address{Mathematical Institute\\
Tohoku University\\
6-3\\
Aramaki Aza-Aoba\\
Aoba-ku\\
Sendai\\
980-8578\\
 JAPAN}
\email{ryosuke.takahashi.e1@tohoku.ac.jp}


\begin{document}
\maketitle
\begin{abstract}
In this paper, we study the modified $J$-equation introduced by Li--Shi \cite{LS16}. We first show that, on compact K\"ahler manifolds, the solvability of the modified $J$-equation is equivalent to the coercivity of the modified $J$-functional. Motivated by this characterization, we formulate a Nakai--Moishezon type criterion for the existence of solutions to the modified $J$-equation on general compact K\"ahler manifolds. We then verify this conjectural criterion in the case of smooth projective toric varieties. This extends the work of Collins--Sz\'ekelyhidi \cite{CS17} and provides further evidence for the expected algebro-geometric nature of the modified $J$-equation. As a potential application, we combine our results with Delcroix--Jubert \cite{DJ25}. Assuming our conjectural Nakai--Moishezon type criterion holds in general, we obtain a numerical sufficient condition for the existence of extremal K\"ahler metrics on arbitrary compact K\"ahler manifolds.
\end{abstract}
\tableofcontents
%==============Section 1==========================
\section{Introduction}
Let $X$ be an $n$-dimensional compact complex manifold and $\Aut_{\rm red}(X)$ be its reduced automorphism group, \ie the connected subgroup of $\Aut(X)$ whose Lie algebra consists of holomorphic vector fields which have a zero on $X$. Let $T \subset \Aut_{\rm red}(X)$ a real torus with $\ft:=\Lie(T)$. Let $\omega$, $\hat{\chi}$ be $T$-invariant K\"ahler forms, and denote their K\"ahler classes by $\a$ and $\b$ respectively. Let $\cH(X,\hat{\chi})^T$ be the space of $T$-invariant $\hat{\chi}$-K\"ahler potentials
\[
\cH(X,\hat{\chi})^T:=\big\{\phi \in C^\infty(X;\R)^T \big| \chi_\phi:=\hat{\chi}+\dd \phi>0 \big\}.
\]
Since $T \subset \Aut_{\rm red}(X)$, the $T$-action on $(X,\hat{\chi})$ is Hamiltonian (\cf \cite[Chapter III, Corollary 4.6]{Kob95}). For any holomorphic vector field $v$ with $\Im(v) \in \ft$, we define the Hamiltonian $\theta_v^X(\chi_\phi)$ as a real-valued smooth function uniquely determined by the properties
\begin{equation} \label{definition and normalization of Hamiltonians}
i_v \chi_\phi=\frac{\i}{2 \pi} \bp \theta_v^X(\chi_\phi), \quad \int_X \theta_v^X(\chi_\phi) \chi_\phi^n=0.
\end{equation}

\subsection{Background and motivation}
Finding canonical metrics has been one of the central themes in K\"ahler geometry. Equations such as the constant scalar curvature K\"ahler (cscK) equation, the K\"ahler--Einstein equation, and their various twisted or weighted analogues reveal deep interactions between complex geometry, algebraic stability, and variational principles (\eg \cite{BJ25,FO19,Lah19,Li21,Mab03,TZ00,Zhu00}). Among these, the $J$-equation, introduced by Donaldson \cite{Don03} and Chen \cite{Che00}, plays a crucial role in understanding the relation between energy properness and the existence of cscK metrics. It also serves as an important model in the continuity approach to cscK metrics \cite{CC18, CC21a, CC21b}. An important progress in the analysis of the $J$-equation was achieved by Song--Weinkove \cite{SW08} who proved that the existence of a subsolution is equivalent to the solvability of the equation. Building on this insight, Collins--Sz\'ekelyhidi \cite{CS17} established a fundamental correspondence between the solvability of the $J$-equation and the coercivity of the associated functional, providing a model for the continuity method in K\"ahler geometry. A further breakthrough was made by G.~Chen \cite{Che21}, who established a connection between the solvability of the $J$-equation and algebro-geometric stability. Using Chen's result \cite{Che21}, Song \cite{Son20} obtained a Nakai--Moishezon type numerical criterion for the solvability of the $J$-equation, confirming the Lejmi--Sz\'ekelyhidi conjecture \cite{LS15}.

However, the usual $J$-equation does not incorporate the effect of holomorphic automorphisms, which often obstruct the existence or uniqueness of canonical metrics. To address this, Li--Shi \cite{LS16} introduced the modified $J$-equation :
\begin{equation} \label{modified J-equation}
\omega \wedge \chi_\phi^{n-1}=\frac{1}{n}(c_X+\theta_v^X(\chi_\phi)) \chi_\phi^n,
\end{equation}
where the constant $c_X$ only depends on $\a$, $\b$, and determined by
\begin{equation} \label{constant cX}
c_X:=\frac{n \a \cdot \b^{n-1}}{\b^n}.
\end{equation}
Also we note that the constant
\begin{equation} \label{constant mX}
m_X:=c_X+\min_X \theta_v^X(\chi_\phi)
\end{equation}
is independent of the choice of $\phi \in \cH(X,\hat{\chi})^T$, and the existence of a solution to \eqref{modified J-equation} easily leads to $m_X>0$ (see Section \ref{Hamiltonians}). They proved that the existence of a subsolution $\underline{\phi} \in \cH(X,\hat{\chi})^T$:
\begin{equation} \label{subsolution}
\big(c_X+\theta_v^X(\chi_{\underline{\phi}}) \big) \chi_{\underline{\phi}}^{n-1}-(n-1) \omega \wedge \chi_{\underline{\phi}}^{n-2}>0
\end{equation}
is equivalent to the solvability of the equation \eqref{modified J-equation}. This equivalence was obtained by studying the modified $J$-flow
\begin{equation} \label{modified J-flow}
\ddt \phi_t=-\Tr_{\chi_{\phi_t}} \omega+c_X+\theta_v^X(\chi_{\phi_t})
\end{equation}
with initial data $\phi_0 \in \cH(X,\hat{\chi})^T$. The modified $J$-equation is closely related to the study of extremal K\"ahler metrics, since the latter can be viewed as critical points of the modified $K$-energy. In this sense, the modified $J$-theory may serve as an analytic model for understanding extremal and, more generally, weighted cscK metrics (see Section \ref{Relation to extremal Kahler metrics}).

\subsection{Main results}
The purpose of this paper is to establish the analytic and numerical foundations of the modified $J$-theory, and to clarify how it extends the previous framework developed by \cite{CS17,LS16}. For later arguments, it is convenient to extend \eqref{modified J-equation} to a more general equation of the form:
\begin{equation} \label{generalized modified J-equation}
\Tr_{\chi_\phi} \omega+b \frac{\omega^n}{\chi_\phi^n}=c+\theta_v^X(\chi_\phi),
\end{equation}
where the constants $b \in \R$, $c>0$ are related with each other by
\begin{equation} \label{relation of b and c}
b=\frac{c \b^n-n \a \cdot \b^{n-1}}{\a^n}.
\end{equation}
In most cases, we assume that $b \geq 0$, but sometimes allow $b$ to be slightly negative. We extends the result \cite[Theorem 3.3]{LS16} to the generalized equation \eqref{generalized modified J-equation} as follows:

\begin{thm} \label{existence of solutions to the generalized equation}
Let $X$ be a compact complex manifold, and $T$, $v$, $\omega$, $\hat{\chi}$ as above. Then the generalized equation \eqref{generalized modified J-equation} with $b \geq 0$ admits a solution if and only if there exists $\underline{\phi} \in \cH(X,\hat{\chi})^T$ such that
\begin{equation} \label{subsolution for generalized equation}
\big(c+\theta_v^X(\chi_{\underline{\phi}}) \big) \chi_{\underline{\phi}}^{n-1}-(n-1) \omega \wedge \chi_{\underline{\phi}}^{n-2}>0.
\end{equation}
\end{thm}

Our first main result is an analogue of the existence-coercivity correspondence proved for the usual $J$-equation in \cite{CS17}.

\begin{thm} \label{existence and coercivity}
Let $X$ be a compact complex manifold, and $T$, $v$, $\omega$, $\hat{\chi}$ as above. Then the modified $J$-equation \eqref{modified J-equation} admits a solution $\phi \in \cH(X,\hat{\chi})^T$ if and only if the modified $J$-functional $J_v^\omega$ is coercive.
\end{thm}

For the precise definition of $J_v^\omega$ and coercivity, see Section \ref{functionals on the space of Kahler potentials}. This result shows that the coercivity of the modified $J$-functional completely characterizes solvability, in the same spirit as the coercivity criterion for various canonical metrics (\eg \cite{BDL20,CC18,CC21a,CC21b,CL21,DR17,DT92,Tia97a}).

In practice, it seems to be difficult to produce $\underline{\phi} \in \cH(X,\hat{\chi})^T$ satisfying \eqref{subsolution}. In particular, it is not clear whether the existence of a solution to the equation depends only on the classes $\a$, $\b$. The following theorem settles this question by showing that solvability depends only on the class $\a$.

\begin{thm} \label{solvability is independent of omega}
Let $X$ be a compact complex manifold, and $T$, $v$, $\omega$, $\hat{\chi}$ as above. Suppose that there exists a $T$-invariant K\"ahler form $\chi \in \b$ satisfying \eqref{modified J-equation} with respect to $\omega$ and $\omega' \in \a$ is another K\"ahler form. Then there exists $\chi' \in \b$ that solves \eqref{modified J-equation} with respect to $\omega'$.
\end{thm}

On the other hand, there is the question of whether the existence of solutions to the equation can be characterized numerically, or algebraically. For instance, the Nakai--Moishezon criterion \cite{DP04} allows one to determine the K\"ahler condition in terms of intersection numbers with arbitrary subvarieties. Theorem \ref{solvability is independent of omega} strongly indicates that a similar criterion should hold for the modified $J$-equation as well. More precisely, we would like to propose the following:

\begin{conj} \label{NM criterion}
Let $X$ be a compact complex manifold, and $T$, $v$, $\omega$, $\hat{\chi}$ as above. Then there exists $\phi \in \cH(X,\hat{\chi})^T$ satisfying \eqref{modified J-equation} if and only if $m_X>0$ and for all $T$-invariant $p$-dimensional irreducible subvarieties $Y \subset X$ with $p=1,\ldots,n-1$ we have
\begin{equation} \label{condition for subvarieties}
\int_Y \big( \big( c_X+\theta_v^X(\chi) \big) \chi^p-p \omega \wedge \chi^{p-1} \big)>0,
\end{equation}
where $\chi \in \b$ is any $T$-invariant K\"ahler form.
\end{conj}

\begin{rk}
The condition \eqref{condition for subvarieties} is independent of the choice of $\chi \in \b$. Indeed, the term including the Hamiltonian $\theta_v^X(\chi)$ can be rewritten as
\[
\int_Y \theta_v^X(\chi) \chi^p=\frac{1}{p+1} \int_Y(\chi+\theta_v^X(\chi))^{p+1}.
\]
Set $d_v:=d-4 \pi i_{\Im(v)}$ so that $\chi+\theta_v^X(\chi)$ is $d_v$-closed. For another $T$-invariant K\"ahler form $\chi'=\chi+\dd \phi$ with a $T$-invariant potential $\phi$, a direct computation shows that
\[
\chi'+\theta_v^X(\chi')=\chi+\theta_v^X(\chi)+\frac{1}{4 \pi} d_v d^c \phi
\]
where $d^c:=\i(\bp-\p)$. So the desired statement follows from the equivariant Chern--Weil theory (\cf \cite{BGV92}).
\end{rk}

\begin{rk}
In particular, when $T$ is trivial and $n=2$, the modified $J$-equation \eqref{modified J-equation} reduces to a complex Monge-Amp\`ere equation:
\[
(c_X \chi_\phi-\omega)^2=\omega^2.
\]
Consequently, Conjecture \ref{NM criterion} follows directly from the criterion of Demailly--P\u{a}un \cite{DP04} and \cite{Yau78}. In contrast, when $T$ is non-trivial, even for surfaces ($n=2$), the modified $J$-equation \eqref{modified J-equation} does not reduce to a complex Monge-Amp\`ere equation, and Conjecture \ref{NM criterion} remains genuinely nontrivial.
\end{rk}

In the case where $T$ is trivial, this was proposed by Lejmi--Sz\'ekelyhidi \cite{LS15} and subsequently established by Song \cite{Son20}, using \cite{Che21} as a key ingredient. We also note that, when the condition \eqref{condition for subvarieties} fails, the formation of singularities of the modified $J$-flow \eqref{modified J-flow} in the Calabi ansatz has been studied by Sivaram \cite{Siv24}. Focusing on the toric setting allows for more explicit results. Theorem \ref{NM criterion for smooth projective toric varieties} confirms Conjecture \ref{NM criterion} in this case, providing a concrete demonstration of the expected behavior under these assumptions.

\begin{thm} \label{NM criterion for smooth projective toric varieties}
Let $X$ be a smooth projective toric variety, $T$ the real torus associated with $X$, and $v$, $\omega$, $\hat{\chi}$ as above. Then there exists a solution $\phi \in \cH(X,\hat{\chi})^T$ to the modified $J$-equation \eqref{modified J-equation} if and only if $m_X>0$ and for all $p$-dimensional irreducible toric subvarieties $Y \subset X$ with $p=1,\ldots,n-1$ we have
\[
\int_Y \big( \big( c_X+\theta_v^X(\chi) \big) \chi^p-p \omega \wedge \chi^{p-1} \big)>0,
\]
where $\chi \in \b$ is any $T$-invariant K\"ahler form.
\end{thm}

As in \cite[Theorem 3]{CS17}, the proof proceeds by analyzing the behavior outside the union of toric divisors $D=\cup_i D_i$, adapting their method to the present context. Unlike \cite{CS17}, we need to control the Hamiltonian near each $D_i$ during the gluing process, by using a $T$-equivariant linearization of $\cO_X(D_i)$. For this reason, we make the additional assumption that $X$ is projective.

\subsection{Relation to extremal K\"ahler metrics} \label{Relation to extremal Kahler metrics}
Recall that extremal K\"ahler metrics are precisely the critical points of the modified $K$-energy $K_{v_{\rm ext}}$, which can be decomposed into the relative entropy and the modified $J$-functional. This decomposition is exactly the Chen--Tian formula \cite{Che00,Tia97a,Tia97b}:
\[
K_{v_{\rm ext}}=H+J_{-v_{\rm ext}}^{-\Ric(\hat{\chi})}.
\]
Using this formula, Li--Shi \cite{LS16} established a sufficient condition related to the modified $J$-equation under which the modified $K$-energy is coercive. Building on this perspective, Delcroix--Jubert \cite{DJ25} further showed that the $T^\C$-coercivity of the modified $K$-energy can be reduced to the existence problem for the modified $J$-equation. This reduction is a key step in connecting extremal K\"ahler metrics with solutions of the modified $J$-equation. Their work, however, addresses a more general setting, namely so-called weighted cscK metrics, which have recently attracted considerable attention in K\"ahler geometry (\eg \cite{AJL23,BJ25,HL24,Lah19,NJL25}). It is natural to expect that the theory of the modified $J$-equation can be further adapted in accordance with their framework, potentially yielding applications to the existence problem for such general canonical metrics. We leave a systematic exploration of such general cases to future work. At the same time, the combination of our results with \cite{DJ25} provides a framework for potential applications to extremal K\"ahler metrics. We explore these applications and related discussions in Section \ref{A numerical criterion for extremal Kahler metrics}.

\subsection{Contributions}
The main contributions of this paper concern the theory of the modified $J$-equation, with the principal new analytic ingredients being the local smoothing and gluing. Our analysis builds on the techniques developed in \cite{Che21, CS17}, which we further adapt to handle the lack of $T$-invariance and the degeneracy of the automorphism flow generated by the holomorphic vector field $v$. These issues form a major obstacle to the construction of subsolutions through local smoothing and gluing procedures. To overcome them, we employ several analytic tools. The first contribution consists of two techniques: the introduction of averaging with respect to the torus action, which restores $T$-invariance, and a careful control of the Hamiltonian near $T$-invariant subvarieties. These together make it possible to carry out the gluing procedures in a way compatible with the $T$-action (\cf Section \ref{Subsolutions} and Theorem \ref{subsolutions near subvarieties}).

The second is the {\it annulus trick}, used in the construction of solutions in Theorem \ref{existence and coercivity}. When establishing the local smoothing estimate, the key is to work in local coordinates that are compatible with the automorphism flow generated by the holomorphic vector field $v$, so that the modified structure is preserved. However, such coordinates can not be defined consistently near the zero set of $v$, where the automorphism action degenerates. To avoid this, we consider the product space $X \times A$, where $A=\{\tau \in \C \mid 1/2<|\tau|<2\}$ is an annulus equipped with a rotational vector field $w=\i \tau \frac{\p}{\p \tau}$. By lifting $v$ to $X \times A$ and adding $w$, we obtain a flow without fixed points. This extended flow allows us to define local coordinates adapted to it, which are then used to perform smoothing and gluing constructions. Combined with an averaging argument along the rotational direction, this yields a subsolution on the slice $X \simeq X \times \{1\} \subset X \times A$. This trick is not merely a computational convenience; it is essential in that it allows us to bypass the highly nontrivial analytic problem of establishing uniform gradient estimates along the modified $J$-flow (see Section \ref{Local smoothing} for more details).

\subsection{Organization of the paper}
The organization of this paper is as follows. In Section \ref{Preliminaries} we review several basic properties that will be used throughout the paper. Section \ref{The modified J-flow} is devoted to the study of the long-time behavior of the modified $J$-flow; in particular, we prove its convergence when $b \geq 0$ under the existence of subsolutions, thereby establishing Theorem \ref{existence of solutions to the generalized equation}. In Section \ref{Coercivity of the modified J-functional}, we show that the existence of solutions is equivalent to the coercivity of the functional $J_v^\omega$ (Theorem \ref{existence and coercivity}), from which Theorem \ref{solvability is independent of omega} follows as a corollary. Section \ref{A Nakai-Moishezon type criterion for smooth projective toric varieties} develops a Nakai--Moishezon type criterion for smooth projective toric varieties, providing the proof of Theorem \ref{NM criterion for smooth projective toric varieties}. Finally, in Section \ref{A numerical criterion for extremal Kahler metrics}, we discuss how Conjecture \ref{NM criterion} can be combined with the results of Delcroix--Jubert \cite{DJ25} to address the existence problem for extremal K\"ahler metrics.

\begin{ackn}
This author was supported by JSPS KAKENHI Grant Numbers 20K14308, 24K06730. The author is also greatful to Vamsi Pritham Pingali for pointing out an error of the proof of Theorem \ref{existence and coercivity} in the previous version.
\end{ackn}

%==============Section 2==========================
\section{Preliminaries} \label{Preliminaries}
\subsection{Convexity properties} \label{Convexity properties} Many of the results in this subsection can be found in \cite[Section 2]{CS17}. Let $\Gamma \subset \R^n$ be the positive orthant and
\[
S_k(\lambda):=\sum_{1 \leq j_1<\ldots<j_k \leq n} \lambda_{j_1} \ldots \lambda_{j_k}, \quad \lambda \in \Gamma
\]
the elementary symmetric function of degree $k$. We have $S_0(\lambda)=1$ and $S_{-1}(\lambda)=0$. For any constant $b \in \R$ we define
\[
f_b(\lambda):=\frac{S_{n-1}(\lambda)}{S_n(\lambda)}+b \frac{S_0(\lambda)}{S_n(\lambda)}=\sum_{i=1}^n \frac{1}{\lambda_i}+\frac{b}{\lambda_1 \cdots \lambda_n}, \quad \lambda \in \Gamma.
\]
Let $\cM$ be the space of all $n \times n$ positive definite Hermitian matrices. Set
\[
F_b(A):=f_b(\lambda),
\]
where $\lambda=(\lambda_1,\ldots,\lambda_n)$ denotes the eigenvalues of $A$ (up to ordering). Then at a diagonal matrix $A \in \cM$ with eigenvalues $\lambda_i$, the derivative of $F_b$ is given by
\begin{equation} \label{derivative formulas}
\p_{ij} F_b=\p_i f_b \cdot \d_{ij}, \quad \p_{ij} \p_{rs} F_b=\p_i \p_r f_b \cdot \d_{ij} \d_{rs}+\frac{\p_i f_b-\p_j f_b}{\lambda_i-\lambda_j}(1-\d_{ij})\d_{ir}\d_{js},
\end{equation}
where we will regard the quotient appeared in the last term as the limit when $\lambda_i=\lambda_j$ \cite{And94,Ger96}. Also for $A \in \cM$ we define
\[
P(A):=\max_{k=1,\ldots,n} \sum_{i \neq k} \frac{1}{\lambda_k}, \quad Q(A):=F_0(A)=\sum_{i=1}^n \frac{1}{\lambda_i}, \quad R(A):=\max_I \sum_{i \in I} \frac{1}{\lambda_i},
\]
where the maximum in $R(A)$ is taken over all subsets $I \subset \{1,\ldots,n\}$ satisfying $|I|=n-2$. For $K>0$ we define convex subsets $\Gamma_K:=\{\lambda \in \Gamma|f_0(\lambda)<K\}$ and $\cM_K:=\{A \in \cM|Q(A)<K \}$. Although $P$ and $R$ are not necessarily smooth except at points where there is a spectral gap, they can be approximated by smooth convex functions, and hence are convex.

\begin{lem}
The functions $P$ and $R$ are convex on $\cM$.
\end{lem}
\begin{proof}
The proof is essentially the same as that of \cite[Proposition 2.2 (4)]{CLT24}. We define a continuous, symmetric, and convex function $g \colon \Gamma \to \R$ by $g(\lambda):=\max_{k=1,\ldots,n} \sum_{i \neq k} \frac{1}{\lambda_i}$. For any fixed $A, A' \in \cM$, we take $K>0$ so that $A, A' \in \cM_K$. Let us consider the convex subset $\Gamma_K \subset \Gamma$ corresponding to $\cM_K$. If $\lambda \in \Gamma_K$, then $\lambda_i>K^{-1}$ holds for all $i$. Thus, for any $r \in (0,\frac{1}{2}K^{-1})$, we can define the smoothing $g^{(r)}$ of $g$ at scale $r$ by
\[
g^{(r)}(\lambda):=\int_{\R^n} r^{-n} \rho \bigg( \frac{|\mu|}{r} \bigg) g(\lambda-\mu) d\mu, \quad \lambda \in \Gamma_K,
\]
where $\rho(t)$ is a smooth, non-negative function supported on $[0,1]$, constant on $[0,1/2]$, and normalized so that $\int_{B_1(0)} \rho(|\mu|)d\mu=1$. Then $g^{(r)}$ is a smooth function on $\Gamma_K$ and converges uniformly to $g$ on $\Gamma_K$ as $r \to 0$. For any $\lambda,\lambda' \in \Gamma_K$ and $t \in [0,1]$, the convexity of $g$ yields that
\[
\begin{aligned}
g^{(r)}(t\lambda+(1-t)\lambda')&=\int_{\R^n} r^{-n} \rho \bigg( \frac{|\mu|}{r} \bigg) g(t \lambda+(1-t)\lambda'-\mu) d\mu \\
&=\int_{\R^n} r^{-n} \rho \bigg( \frac{|\mu|}{r} \bigg) g(t (\lambda-\mu)+(1-t)(\lambda'-\mu)) d\mu \\
&\leq \int_{\R^n} r^{-n} \rho \bigg( \frac{|\mu|}{r} \bigg) \big( t g(\lambda-\mu)+(1-t)g(\lambda'-\mu) \big) d\mu \\
&=tg^{(r)}(\lambda)+(1-t)g^{(r)}(\lambda'),
\end{aligned}
\]
which shows that $g^{(r)}$ is convex on $\Gamma_K$. Let $\mathfrak{S}_n$ be the symmetric group of degree $n$ which acts linearly on $\Gamma_K$ by $\s \cdot \lambda:=(\lambda_{\s(1)},\ldots,\lambda_{\s(n)})$. Then the symmetrization $g_r(\lambda):=\frac{1}{n!} \sum_{\s \in \mathfrak{S}_n} g^{(r)}(\s \cdot \lambda)$ is convex and converges uniformly to $g$ on $\Gamma_K$ as $r \to 0$. Therefore, we can apply the result \cite{Spr05} to $g_r$, and deduce that the function $P_r$ on $\cM_K$ corresponding to $g_r$ is convex. Thus we have
\[
P_r(tA+(1-t)A') \leq tP_r(A)+(1-t)P_r(A'), \quad t \in [0,1].
\]
Taking the limit $r \to 0$, we obtain
\[
P(tA+(1-t)A') \leq tP(A)+(1-t)P(A'), \quad t \in [0,1].
\]
Since $A$ and $A'$ are arbitrarily chosen in $\cM$, this shows that $P$ is convex on $\cM$. The same argument applies to $R$ as well.
\end{proof}

For K\"ahler forms $\omega$, $\chi$ on a complex manifold $X$, we often write $g$, $h$ as the Riemannian metric corresponding to $\omega$, $\chi$ respectively, \ie
\[
\omega=\frac{\i}{2\pi} \sum_{i,j} g_{i\bar{j}} dz^i \wedge dz^{\bar{j}}, \quad \chi=\frac{\i}{2\pi} \sum_{i,j} h_{i\bar{j}} dz^i \wedge dz^{\bar{j}}
\]
in local coordinates. We define the matrix $A$ by $A_j^i:=g^{i\bar{k}}h_{j\bar{k}}$, and set
\[
F_{\omega,b}(\chi):=F_b(A), \quad P_\omega(\chi):=P(A), \quad Q_\omega(\chi):=Q(A), \quad R_\omega(\chi):=R(A),
\]
where the right hand side of each of the above expression is computed of the pointwise eigenvalues $\lambda_i$ of $A$. It is also worth nothing that the eigenvalues of $A$ do not depend on the choice of local coordinates $(z^1,\ldots,z^n)$, but only on $\omega$ and $\chi$. For a submanifold $Y \subset X$, we often consider the restriction of the operators $P_\omega$, $Q_\omega$ to $Y$:
\[
P_{Y,\omega}(\chi):=P_{\omega|_Y}(\chi|_Y), \quad Q_{Y,\omega}(\chi):=Q_{\omega|_Y}(\chi|_Y).
\]
Then we have the following:
\begin{prop} \label{restriction formulas}
\begin{enumerate}
\item For any submanifold $Y \subset X$ we have
\[
Q_{Y,\omega}(\chi) \leq Q_{X,\omega}(\chi).
\]
\item For any hypersurface $Y \subset X$ we have
\[
P_{Y,\omega}(\chi) \leq R_{X,\omega}(\chi).
\]
\end{enumerate}
\end{prop}
\begin{proof}
This follows immediately from the Courant--Fischer--Weyl type min-max principle established in \cite[Section 3]{Che21} (see also \cite[Section 2]{CLT24}), but for the reader’s convenience, we provide a detailed proof here.

Let $p:=\dim Y$. At each point $x \in Y$, we choose an $\omega$-orthonormal basis $e_1,\ldots,e_n$ of $T_x X$ and identify $T_x X \simeq \C^n$. Moreover, we choose the basis such that $e_1,\ldots,e_p$ span $T_x Y$, so that we obtain $T_x Y \simeq \C^p \simeq \C^p \times \{0\} \subset \C^n$. This allows us to identify $\chi$ with an $n \times n$ positive definite Hermitian matrix $A$. We shall regard $A$ as a bilinear form on $\C^n$.

For any positive integer $k \leq p$ and $k$-dimensional subspace $V$ of $\C^n$ there exists $n \times k$ matrix $U \in \Mat(n,k;\C)$ such that $\overline{U}^T U=I_k$ and a basis of $V$ consists of the columns of $U$. The representation of the restriction $A|_V$ with respect to the basis $U$ is then given by $\overline{U}^TAU$. In particular, we note that $\Tr (A|_V)^{-1}:=\Tr (\overline{U}^TAU)^{-1}$ is independent of the choice of $U$ and depends only on $A$ and $V$. Also if $V \subset \C^p$ then $U$ can be decomposed as
\[
U=
\begin{pmatrix}
U'\\
0
\end{pmatrix}
\]
for a matrix $U' \in \Mat(p,k;\C)$ such that $\overline{U'}^T U'=I_k$. Thus denoting by $A'$ the $p \times p$ upper-left block of $A$, we obtain $\overline{U}^TAU=\overline{U'}^TA'U'$, which shows that the definition of $\Tr(A|_V)^{-1}$ is compatible with restriction to $\C^p$.

\noindent
(1) We choose $U$ such that $\overline{U}^TAU=\diag(\lambda_1',\ldots,\lambda_p')$ and $S \in U(n)$ such that $S_{i \bar{j}}=U_{i \bar{j}}$ ($i=1,\ldots,n$; $j=1,\ldots,p$). Then by using the Schur--Horn theorem we know that the diagonal $(\lambda_1',\ldots,\lambda_p',(\overline{S}^TAS)_{p+1,\overline{p+1}},\ldots,(\overline{S}^TAS)_{n,\overline{n}})$ of the matrix $\overline{S}^TAS$ lies in the convex full of the vectors obtained by permuting the entries of $(\lambda_1,\ldots,\lambda_n)$, where $\lambda_i$ are eigenvalues of $A$. Let $I \subset \{1,\ldots,n\}$ be a subset with $|I|=p$. Combining with the convexity of the function $\Gamma \ni (\mu_1,\ldots,\mu_n) \mapsto \max_I \sum_{i \in I} \frac{1}{\mu_i} \in \R$ we know that
\[
\begin{aligned}
Q_{X,\omega}(\chi)&=\sum_{i=1}^n \frac{1}{\lambda_i} \geq \max_I \sum_{i \in I} \frac{1}{\lambda_i}=\max_{\substack{U \in \Mat(n,p;\C);\\ \overline{U}^TU=I_p, \;\overline{U}^TAU=\diag(\lambda_1',\ldots,\lambda_p')}} \sum_{i=1}^p \frac{1}{\lambda_i'}=\max_{V \subset \C^n} \Tr(A|_V)^{-1}\\
&\geq \Tr(A|_{\C^p})^{-1}=Q_{Y,\omega}(\chi),
\end{aligned}
\]
where the maximum is taken over all $p$-dimensional subspaces $V \subset \C^n$.

\noindent
(2) As in (1), we can deduce that
\[
R_{X,\omega}(\chi)=\max_{V \subset \C^n} \Tr(A|_V)^{-1} \geq \max_{V \subset \C^{n-1}} \Tr(A|_V)^{-1}=P_{Y,\omega}(\chi),
\]
where the maxima are taken over all $(n-2)$-dimensional subspaces $V \subset \C^n$ or $V \subset \C^{n-1}$.
\end{proof}

\subsubsection{The $b=-\e<0$ case}
Next we consider the case $b=-\e<0$.

\begin{prop} \label{properties for non-positive case}
For any $K>0$ there exists $\e_1=\e_1(n,K)>0$ such that for all $\e \in (0,\e_1)$, the function $f_{-\e} \colon \Gamma_K \to \R$ satisfies the following properties:
\begin{enumerate}
\item $f_{-\e}>0$.
\item $\p_i f_{-\e}<0$ for all $i$.
\item If $\lambda_i \geq \lambda_j$, then $\p_i f_{-\e} \geq \p_j f_{-\e}$.
\item $f_{-\e}$ is convex.
\end{enumerate}
\end{prop}
\begin{proof}
For notational convenience let
\[
S_{n-1;i}(\lambda):=\frac{S_n(\lambda)}{\lambda_i}.
\]
Since $\lambda_i>K^{-1}$ from the condition $f_0(\lambda)<K$ we have
\[
f_{-\e}(\lambda)=\frac{S_{n-1}(\lambda)-\e}{S_n(\lambda)} \geq \frac{(n-1)K^{-n+1}-\e}{S_n(\lambda)},
\]
\[
\p_i f_{-\e}(\lambda)=\frac{-S_{n-1;i}(\lambda)+\e}{\lambda_i S_n(\lambda)} \leq \frac{-K^{-n+1}+\e}{\lambda_i S_n(\lambda)}.
\]
Also if $\lambda_i \geq \lambda_j$ then
\[
\begin{aligned}
\p_i f_{-\e}(\lambda)-\p_j f_{-\e}(\lambda)&=\frac{(\lambda_i-\lambda_j)(S_{n-1;i}(\lambda)+S_{n-1;j}(\lambda)-\e)}{\lambda_i \lambda_j S_n(\lambda)} \\
& \geq \frac{(\lambda_i-\lambda_j)(2K^{-n+1}-\e)}{\lambda_i \lambda_j S_n(\lambda)}.
\end{aligned}
\]
Finally, we compute
\[
\p_i \p_j f_{-\e}(\lambda)=\frac{2 S_{n-1;i}(\lambda) \d_{ij}-\e (1+\d_{ij})}{\lambda_i \lambda_j S_n(\lambda)}.
\]
Thus for all $w=(w_1,\ldots,w_n) \in \R^n$ we have
\[
\begin{aligned}
\sum_{i,j} w_i w_j \p_i \p_j f_{-\e}(\lambda)&=\frac{2}{S_n(\lambda)} \sum_i \frac{w_i^2}{\lambda_i^2} S_{n-1;i}(\lambda)-\frac{\e}{S_n(\lambda)} \bigg( \sum_i \frac{w_i}{\lambda_i} \bigg)^2-\frac{\e}{S_n(\lambda)} \sum_i \frac{w_i^2}{\lambda_i^2} \\
& \geq \frac{2K^{-n+1}-\e (n+1)}{S_n(\lambda)} \sum_i \frac{w_i^2}{\lambda_i^2}.
\end{aligned}
\]
The above computations show that all of the required statements hold if $\e$ is sufficiently small.
\end{proof}
Thus combining with \eqref{derivative formulas}, we know that $F_{-\e} \colon \cM_K \to \R$ is convex for $\e \in (0,\e_1)$ (\cf \cite{And94,Ger96,Spr05}). However, we need a stronger convexity property as follows:

\begin{prop}[\cite{CS17}, Lemma 9] \label{strong convexity}
For any $K>0$ there exists $\e_2(n,K)>0$ such that for all $\e \in (0,\e_2)$, diagonal matrix $A \in \cM_K$ with entries $\lambda_i$ and Hermitian matrix $B_{ij}$ we have
\[
\sum_{p,q,r,s} B_{rs} \overline{B_{qp}} \p_{pq} \p_{rs} F_{-\e}(A)+\sum_{i,j}|B_{ij}|^2 \frac{\p_{ii} F_{-\e}(A)}{\lambda_j} \geq 0.
\]
\end{prop}

We also need the following estimate to deal with the modified $J$-flow with negative twist $b=-\e<0$.

\begin{lem} \label{estimate along continuous path}
For any $K>0$ and $C>0$ there exists a constant $\e_3=\e_3(n,K,C)>0$ such that for any $\e \in (0,\e_3)$, we have the following: if the continuous path $A_t \in \cM$ ($t \in [0,1]$) satisfies
\[
Q(A_0)<K, \quad F_{-\e}(A_t)<K+2C \quad (t \in [0,1]),
\]
then $Q(A_t)<K+3C$ for $t \in [0,1]$.
\end{lem}
\begin{proof}
We follow closely to the argument \cite[Lemma 19]{CS17}. Let $\lambda_i$ be the eigenvalue of $A_t$. Assume that $Q(A_t)=K+3C$ for some $t \in (0,1]$. Then AM-GM inequality shows that
\[
\frac{1}{\lambda_1 \cdots \lambda_n} \leq n^{-n} \bigg( \sum_{i=1}^n \frac{1}{\lambda_i} \bigg)^n=n^{-n}(K+3C)^n.
\]
If we set $\e_3:=\frac{n^n C}{2(K+3C)^n}$ then for any $\e \in (0,\e_3)$ we have
\[
F_{-\e}(A_t)=\sum_{i=1}^n \frac{1}{\lambda_i}-\frac{\e}{\lambda_1 \cdots \lambda_n} \geq K+3C-\e n^{-n}(K+2C)^n>K+2C
\]
which contradicts the assumption. So we get $Q(A_t)<K+3C$ for all $t \in [0,1]$.
\end{proof}

\begin{lem} \label{change of omega}
Let $\omega_1$, $\omega_2$, $\chi$ be K\"ahler forms on $X$. Assume that $Q_{\omega_2}(\chi) \leq K$ and
\[
(1-\s) \omega_1 \leq \omega_2 \leq (1+\s) \omega_1
\]
for some $K>0$ and $0<\s<1/2$. Then we have
\[
|Q_{\omega_1}(\chi)-Q_{\omega_2}(\chi)| \leq 2K \s, \quad |R_{\omega_1}(\chi)-R_{\omega_2}(\chi)| \leq 2K \s.
\]
\end{lem}
\begin{proof}
Let $\lambda_1 \leq \ldots \leq \lambda_n$ and $\mu_1 \leq \ldots \leq \mu_n$ be eigenvalues of $\chi$ with respect to $\omega_1$ and $\omega_2$ respectively. Then the assumption yields that
\[
(1-\s) \mu_i \leq \lambda_i \leq (1+\s) \mu_i, \quad i=1,\ldots,n.
\]
From the assumption $Q_{\omega_2}(\chi) \leq K$ we get
\[
|Q_{\omega_1}(\chi)-Q_{\omega_2}(\chi)| \leq \sum_{i=1}^n \bigg| \frac{1}{\lambda_i}-\frac{1}{\mu_i} \bigg|=\sum_{i=1}^n \frac{|\lambda_i-\mu_i|}{\lambda_i \mu_i} \leq \frac{\s}{1-\s} \sum_{i=1}^n \frac{1}{\mu_i} \leq 2K \s.
\]
The proof for $R$ is similar.
\end{proof}

\begin{lem} \label{F to P}
Let $\omega$ and $\chi$ be K\"ahler forms and $K>0$ a constant. Assume that $Q_\omega(\chi) \leq K$. Then there exists $\e_4=\e_4(n,K)>0$ such that for all $\e \in (0,\e_4)$ we have
\[
P_\omega(\chi) \leq F_{\omega,-\e}(\chi).
\]
\end{lem}
\begin{proof}
Let $\lambda_1 \leq \ldots \leq \lambda_n$ be the eigenvalues of $\chi$ with respect to $\omega$. By using $\lambda_i \geq K^{-1}$ we compute
\[
F_{\omega,-\e}(\chi)=P_\omega(\chi)+\frac{1}{\lambda_n} \bigg(1-\frac{\e}{\lambda_1 \cdots \lambda_{n-1}} \bigg) \geq P_\omega(\chi)+\frac{1}{\lambda_n}(1-\e K^{n-1}).
\]
So if we set $\e_4=K^{-n+1}$ then for any $\e \in (0,\e_4)$, we have
\[
P_\omega (\chi) \leq F_{\omega,-\e}(\chi).
\]
\end{proof}

\subsubsection{The $b \geq 0$ case}
Now we consider the case $b \geq 0$. In this case, the situation is considerably easier to handle compared to the case $b=-\e<0$.
\begin{prop} \label{properties for semipositive case}
For $b \geq 0$ the function $f_b \colon \Gamma \to \R$ satisfies the following properties:
\begin{enumerate}
\item $f_b>0$.
\item $\p_i f_b<0$ for all $i$.
\item If $\lambda_i \geq \lambda_j$, then $\p_i f_b \geq \p_j f_b$.
\item $f_b$ is convex.
\end{enumerate}
\end{prop}
\begin{proof}
The calculation in Proposition \ref{properties for non-positive case} can be applied verbatim to this case.
\end{proof}

Thus the propeties (3), (4) together with \eqref{derivative formulas} implies that $F_b \colon \cM \to \R$ is convex. A strong convexity property holds in the same manner:

\begin{prop}[\cite{CS17}, Lemma 8] \label{strong convexity for positive b}
For any diagonal matrix $A \in \cM$ with eigenvalues $\lambda_i$ and Hermitian matrix $B_{ij}$ we have
\[
\sum_{p,q,r,s} B_{rs} \overline{B_{qp}} \p_{pq} \p_{rs} F_b(A)+\sum_{i,j}|B_{ij}|^2 \frac{\p_{ii} F_b(A)}{\lambda_j} \geq 0.
\]
\end{prop}

\subsection{Hamiltonians} \label{Hamiltonians}
Let $X$ be a compact complex manifold, $T \subset \Aut_{\rm red}(X)$ a real torus, $\hat{\chi}$ a $T$-invariant K\"ahler form on $X$, and $v$ a holomorphic vector field satisfying $\Im(v) \in \ft$.
\begin{lem} \label{T-invariance and C0-estimate of Hamiltonians}
For any $\chi_\phi=\hat{\chi}+\dd \phi$ with $\phi \in \cH(X,\hat{\chi})^T$, the following statement holds:
\begin{enumerate}
\item The function $\theta_v^X(\chi_\phi)$ is $T$-invariant.
\item We have $\theta_v^X(\chi_\phi)=\theta_v^X(\hat{\chi})+v(\phi)$. In particular, $\max_X \theta_v(\chi_\phi)$ as well as $\min_X \theta_v(\chi_\phi)$ is independent of the choice of $\phi \in \cH(X,\hat{\chi})^T$.
\end{enumerate}
\end{lem}
\begin{proof}
(1) By \eqref{definition and normalization of Hamiltonians}, we have
\[
i_{\Im(v)} \chi_\phi=\frac{1}{4 \pi}d \theta_v^X(\chi_\phi),
\]
which shows that $d \theta_v^X(\chi_\phi)$ is $T$-invariant. So for any $\tau \in T$ we have $d \tau^\ast \theta_v^X(\chi_\phi)=\tau^\ast d \theta_v^X(\chi_\phi)=d \theta_v^X(\chi_\phi)$, from which it follows that $\tau^\ast \theta_v^X(\chi_\phi)-\theta_v^X(\chi_\phi)=C$ for some constant $C \in \R$. By using \eqref{definition and normalization of Hamiltonians} and the fact that $\chi_\phi$ is $T$-invariant, integrating over $X$ with respect to $\chi_\phi^n$ yields $C=0$. Hence $\theta_v^X(\chi_\phi)$ is $T$-invariant.

\noindent
(2) The result is obtained in \cite[Corollary 5.3]{Zhu00}. Applying $i_v$ to $\chi_\phi=\hat{\chi}+\dd \phi$ we observe that $\theta_v^X(\chi_\phi)=\theta_v^X(\hat{\chi})+v(\phi)$ modulo additive constants. Thus
\[
\theta_v^X(\chi_{t\phi})=\theta_v^X(\hat{\chi})+t v(\phi)+C_t, \quad t \in [0,1]
\]
for some $t$-dependent constant $C_t$. Differentiating $\int_X \theta_v^X(\chi_{t\phi}) \chi_{t \phi}^n=0$ in $t$ and integrating by parts yield that
\[
0=\ddt \int_X \theta_v^X(\chi_{t\phi}) \chi_{t \phi}^n=\int_X \big( v(\phi)+\ddt C_t+\theta_v^X(\chi_{t\phi}) \Delta_{t\phi} \phi \big) \chi_{t\phi}^n=\b^n \ddt C_t.
\]
Thus we have $C_t=C_0=0$ which shows the first assertion. For the second assertion we take a point $x_0 \in X$ where $\theta_v(\chi_\phi)$ achieves the maximum. Then $\bp \theta_v(\chi_\phi)(x_0)=0$, so the vector $v$ vanishes at $x_0$. Combining with $\theta_v^X(\chi_\phi)=\theta_v^X(\hat{\chi})+v(\phi)$ we observe that
\[
\max_X \theta_v^X(\chi_\phi)=\theta_v^X(\chi_\phi)(x_0)=\theta_v^X(\hat{\chi})(x_0) \leq \max_X \theta_v^X(\hat{\chi}).
\]
By changing the role of $\hat{\chi}$ and $\chi_\phi$ we obtain $\max_X \theta_v^X(\chi_\phi)=\max_X \theta_v^X(\hat{\chi})$. The proof for $\min_X \theta_v^X(\chi_\phi)$ is similar.
\end{proof}
From the above lemma we know that the constant
\[
m_X=c_X+\min_X \theta_v^X(\chi_\phi)
\]
is independent of a choice of $\phi \in \cH(X,\hat{\chi})^T$. We fix a uniform constant $C_\theta>0$ (independent of $\phi \in \cH(X,\hat{\chi})^T$) such that
\begin{equation} \label{Ctheta}
|\theta_v^X(\chi_\phi)| \leq C_\theta.
\end{equation}

\subsection{Subsolutions} \label{Subsolutions}
In this subsection we study general properties for subsolutions \eqref{subsolution} related to the regularized maximum and averaging over the torus $T$. We first discuss how the subsolution condition can be expressed in terms of the operator $P_\omega$.
\begin{prop}
The subsolution condition \eqref{subsolution for generalized equation} for $\chi:=\chi_{\underline{\phi}}$ is equivalent to
\[
P_\omega(\chi)-\theta_v^X(\chi)<c.
\]
In this case, we have
\begin{equation} \label{p positivity}
\big(c+\theta_v^X(\chi) \big)\chi^p-p\omega \wedge \chi^{p-1}>0
\end{equation}
for all $p=1,\ldots,n-1$.
\end{prop}
\begin{proof}
This follows immediately from the computation in \cite[Section 1]{Che21}. We first note that the positivity condition \eqref{p positivity} (in the sense of \cite[Chapter III, Section 1.A]{Dem12}) means that
\[
\big(c+\theta_v^X(\chi) \big)\chi^p-p\omega \wedge \chi^{p-1} \wedge \Omega>0
\]
for any simple positive $(n-p,n-p)$-form $\Omega$. We choose a coordinates such that $g_{i \bar{j}}=\d_{ij}$ and $h_{i\bar{j}}=\lambda_i \d_{ij}$. Then
\[
\big(c+\theta_v^X(\chi) \big)\chi^p-p\omega \wedge \chi^{p-1}=p! \bigg(\frac{\i}{2 \pi} \bigg)^p \sum_I \bigg(c+\theta_v^X(\chi)-\sum_{i \in I} \frac{1}{\lambda_i} \bigg) \bigg( \prod_{i \in I} \lambda_i \bigg) dz^I \wedge d\bar{z}^I,
\]
where the sum is taken over all subsets $I \subset \{1,\ldots,n\}$ with $|I|=p$. The form $\Omega$ can be expressed as
\[
\Omega=\bigg(\frac{\i}{2 \pi} \bigg)^{n-p} \sum_I a_I dz^{I^c} \wedge d \bar{z}^{I^c}+\Omega',
\]
where $I^c$ denotes the complement of $I \subset \{1,\ldots,n\}$ and $\Omega'$ is a smooth form such that $\Omega' \wedge dz^I \wedge d\bar{z}^I=0$ for all $I$. Since $\Omega$ is a simple positive form, we have $a_I \geq 0$ for all $I$ and at least one $a_I$ is positive. Then a direct computation shows that
\[
\frac{\big(c+\theta_v^X(\chi) \big)\chi^p-p\omega \wedge \chi^{p-1} \wedge \Omega}{\omega^n}=\frac{p!}{n!} \sum_I a_I \bigg(c+\theta_v^X(\chi)-\sum_{i \in I} \frac{1}{\lambda_i} \bigg) \bigg( \prod_{i \in I} \lambda_i \bigg).
\]
The desired statement follows from the above formula.
\end{proof}
Following \cite[Section 5.E]{Dem12} for arbitrary $\eta=(\eta_1,\ldots,\eta_N) \in (0,\infty)^N$, we define the function $M_\eta \colon \R^N \to \R$ by
\[
M_\eta(t_1,\ldots,t_N):=\int_{\R^N} \max\{t_1+h_1,\ldots,t_N+h_N \} \prod_{j=1,\ldots,N} \theta \bigg(\frac{h_j}{\eta_j} \bigg) dh_1 \ldots dh_N,
\]
where $\theta$ is a non-negative smooth function on $\R$, supported on $[-1,1]$, such that $\int_\R \theta(h)dh=1$ and $\int_\R h \theta(h)dh=0$. The fundamental properties for the regularized maximum are summed up as follows:
\begin{lem}[\cite{Dem12}, Lemma 5.18]
The function $M_\eta$ satisfies the following:
\begin{enumerate}
\item $M_\eta(t_1,\ldots,t_N)$ is non decreasing in all variables and convex on $\R^N$.
\item $\max\{t_1,\ldots,t_N \} \leq M_\eta(t_1,\ldots,t_N) \leq \max\{t_1+\eta_1,\ldots,t_N+\eta_N \}$.
\item $M_\eta(t_1,\ldots,t_N)=M_{(\eta_1,\ldots,\check{\eta_j},\ldots,\eta_N)}(t_1,\ldots,\check{t_j},\ldots,t_N)$ if $t_j+\eta_j \leq \max_{k \neq j}\{t_k-\eta_k\}$, where $\check{f}$ means eliminating $f$.
\item $M_\eta(t_1+a,\ldots,t_N+a)=M_\eta(t_1,\ldots,t_N)+a$ for all $a \in \R$.
\end{enumerate}
\end{lem}
By the property (1) and differentiating (4) at $a=0$ we have
\begin{equation} \label{sum is one}
\sum_{j=1} \frac{\p M_\eta}{\p t_j}=1, \quad \frac{\p M_\eta}{\p t_j} \geq 0.
\end{equation}

Now let $X$ be a complex manifold, and let $\omega$ and $\hat{\chi}$ be K\"ahler forms on $X$. Let $\{B_j\}_{j \in \{1,\ldots,N\}}$ be a finite covering of $X$ and $\varphi_j$ smooth functions on $B_j$ satisfying $\varphi_j(x)<\max_{k=1,\ldots,N}\{\varphi_k(x)\}$ at every point $x \in \p B_j$ (where we extend $\varphi_j$ on $X$ so that $\varphi_j=-\infty$ outside $B_j$). In particular, we can choose a sufficiently small vector $\eta=(\eta_1,\ldots,\eta_N) \in \R^N$ so that $\varphi_j(x)+\eta_j \leq \max_{k=1,\ldots,N}\{\varphi_k(x)-\eta_k\}$ holds for all $j$ and $x \in \p B_j$. Then as a corollary of the above lemma (see also \cite[Corollary 5.19]{Dem12}) we find that the function $\varphi:=M_\eta(\varphi_1,\ldots,\varphi_N)$ is smooth on $X$. Moreover, a direct computation shows that the second derivatives in local coordinates $(z^1,\ldots,z^n)$ are given by
\begin{equation} \label{second derivative of Meta}
\frac{\p^2 \varphi}{\p z^k \p z^{\bar{\ell}}}=\sum_{a,b} \frac{\p^2 M_\eta}{\p t_a \p t_b} (\varphi_1,\ldots,\varphi_N) \cdot \frac{\p \varphi_a}{\p z^k} \cdot \frac{\p \varphi_b}{\p z^{\bar{\ell}}}+\sum_a \frac{\p M_\eta}{\p t_a} (\varphi_1,\ldots,\varphi_N) \cdot \frac{\p^2 \varphi_a}{\p z^k \p z^{\bar{\ell}}}.
\end{equation}
The following extends the argument in \cite[Section 4]{Che21}:
\begin{prop} \label{gluing of subsolutions}
In the above setting, let $f \in C^\infty(X;\R)$ and $w$ a holomorphic vector field on $X$. We further assume that $\hat{\chi}+\dd \varphi_j>0$ on each $B_j$. Then for a sufficiently small vector $\eta \in \R^N$ chosen as above, the regularized maximum $\varphi=M_\eta(\varphi_1,\ldots,\varphi_N)$ satisfies $\chi:=\hat{\chi}+\dd \varphi>0$. Moreover, if $\varphi_j$ satisfies
\[
P_\omega \bigg(\hat{\chi}+\dd \varphi_j \bigg)-\Re(w)(\varphi_j)<f
\]
on each $B_j$, then
\[
P_\omega(\chi)-\Re(w)(\varphi)<f.
\]
A similar formula holds when $P_\omega$ is replaced by $Q_\omega$ or $R_\omega$.
\end{prop}
\begin{proof}
\noindent
By the convexity of $M_\eta$, \eqref{sum is one} and \eqref{second derivative of Meta} we observe that
\[
\chi \geq \sum_j \frac{\p M_\eta}{\p t_j} \cdot \bigg( \hat{\chi}+\dd \varphi_j \bigg)>0.
\]
The term $\sum_j \frac{\p M_\eta}{\p t_j} \cdot \big(\hat{\chi}+\dd \varphi_j \big)$ is a weighted average of $\hat{\chi}+\dd \varphi_j$ with $\sum_j \frac{\p M_\eta}{\p t_j}=1$. Thus the monotonicity and convexity property of $P_\omega$ shows that
\[
P_\omega(\chi) \leq \sum_j \frac{\p M_\eta}{\p t_j} \cdot P_\omega \bigg(\hat{\chi}+\dd \varphi_j \bigg),
\]
and hence
\[
P_\omega(\chi)-\Re(w)(\varphi) \leq \sum_j \frac{\p M_\eta}{\p t_j} \cdot \bigg[ P_\omega \bigg(\hat{\chi}+\dd \varphi_j \bigg)-\Re(w)(\varphi_j) \bigg]<f.
\]
The above proof only requires the monotonicity and convexity of the operator $P_\omega$. So the same proof also works for $Q_\omega$ and $R_\omega$.
\end{proof}

In general, the K\"ahler form $\chi$ obtained in Proposition \ref{gluing of subsolutions} is not necessarily $T$-invariant. For this reason we need the averaging argument as follows:

\begin{prop} \label{averaging of potentials and forms}
Let $X$ be a complex manifold, $T \subset \Aut(X)$ a real torus, and let $\omega$ and $\hat{\chi}$ be $T$-invariant K\"ahler forms on $X$. Let $f \in C^\infty(X;\R)^T$ and $w$ a holomorphic vector field on $X$ such that $\Re(w)$ is $T$-invariant. For a $\hat{\chi}$-K\"ahler potential $\varphi$ on $X$, set $\chi:=\hat{\chi}+\dd \varphi>0$,
\[
\tilde{\varphi}:=\int_{\tau \in T} \tau^\ast \varphi dT, \quad \tilde{\chi}:=\hat{\chi}+\dd \tilde{\varphi}
\]
so that $\tilde{\chi}=\int_{\tau \in T} \tau^\ast \chi dT>0$, where $dT$ denotes the Haar measure on $T$ normalized by $\int_T dT=1$. If
\[
P_\omega(\chi)-\Re(w)(\varphi)<f,
\]
then the $T$-invariant potential $\tilde{\varphi}$ satisfies
\[
P_\omega(\tilde{\chi})-\Re(w)(\tilde{\varphi})<f.
\]
A similar formula holds when $P_\omega$ is replaced by $Q_\omega$ or $R_\omega$.
\end{prop}
\begin{proof}
Since $\omega$ is $T$-invariant, the K\"ahler form $\tau^\ast \chi$ can be identified with a positive definite Hermitian matrix by using an $\omega$-orthonormal basis at each point on $X$. So the convexity of the function $P$ implies that
\[
P_\omega(\tilde{\chi}) \leq \int_{\tau \in T} \tau^\ast P_\omega(\chi) dT.
\]
Since $\Re(w)$ and $f$ are $T$-invariant, we obtain
\[
P_\omega(\tilde{\chi})-\Re(w)(\tilde{\varphi}) \leq \int_{\tau \in T} \tau^\ast \big( P_\omega(\chi)-\Re(w)(\varphi) \big) dT<\int_{\tau \in T} \tau^\ast f dT=f.
\]
In the above proof, we only used the facts that $\omega$ is $T$-invariant and $P_\omega$ is convex. Hence the same assertion also holds for $Q_\omega$ and $R_\omega$.
\end{proof}

\begin{prop} \label{averaging of subsolutions}
Let $X$ be a compact complex manifold, $T \subset \Aut_{\rm red}(X)$ a real torus, and let $\omega$ and $\hat{\chi}$ be $T$-invariant K\"ahler forms on $X$. Let $v$ be a holomorphic vector field on $X$ satisfying $\Im(v) \in \ft$. For a $\hat{\chi}$-K\"ahler potential $\varphi$ on $X$, set $\chi:=\hat{\chi}+\dd \varphi>0$. If
\[
P_\omega(\chi)-\theta_v^X(\hat{\chi})-\Re(v)(\varphi)<c
\]
for some constant $c>0$, then the average $\tilde{\chi}:=\hat{\chi}+\dd \tilde{\varphi}$ with $\tilde{\varphi}:=\int_{\tau \in T} \tau^\ast \varphi dT$ satisfies
\[
P_\omega(\tilde{\chi})-\theta_v^X(\tilde{\chi})<c.
\]
A similar formula holds when $P_\omega$ is replaced by $Q_\omega$.
\end{prop}
\begin{proof}
This is a special case of Proposition \ref{averaging of potentials and forms} obtained by setting $w=v$ and $f:=c+\theta_v^X(\hat{\chi})$. Indeed, the $T$-action on $X$ admits a natural complexification, so that $\Re(v)$ is $T$-invariant (\cf \cite[Section 3]{Ish19}).
\end{proof}

\subsection{Functionals on the space of K\"ahler potentials} \label{functionals on the space of Kahler potentials}
Let $X$ be a compact complex manifold, $T \subset \Aut_{\rm red}(X)$ a real torus, and let $\omega \in \a$ and $\hat{\chi} \in \b$ be $T$-invariant K\"ahler forms. Define Aubin's $I$-functional \cite{Aub84} by the formula
\[
I(\phi):=\int_X \phi(\hat{\chi}^n-\chi_\phi^n).
\]
Integrating by parts one can see that
\begin{equation} \label{formula for I}
\begin{aligned}
I(\phi)&=-\int_X \phi \dd \phi \wedge (\chi_\phi^{n-1}+\chi_\phi^{n-2} \wedge \hat{\chi}+\cdots+\hat{\chi}^{n-1}) \\
&=\frac{\i}{2 \pi} \int_X  \p \phi \wedge \bp \phi \wedge (\chi_\phi^{n-1}+\chi_\phi^{n-2} \wedge \hat{\chi}+\cdots+\hat{\chi}^{n-1}) \\
& \geq \int_X |\p \phi|_{\hat{h}}^2 \hat{\chi}^n,
\end{aligned}
\end{equation}
where $\hat{h}$ denotes the Riemannian metric corresponding to $\hat{\chi}$. Let $b \in \R$, $c>0$ be constants related by \eqref{relation of b and c}. We define the modified $J$-functional by
\begin{equation} \label{Jvomegab functional}
J_{v,b}^\omega(\phi):=\int_0^1 \int_X \phi \big( F_{\omega,b}(\chi_{t \phi})-c-\theta_v^X(\chi_{t \phi}) \big) \chi_{t \phi}^n dt, \quad \phi \in \cH(X,\hat{\chi})^T.
\end{equation}
By the definition the critical points of $J_{v,b}^\omega$ are just given by solutions to the generalized equation \eqref{generalized modified J-equation}. When $b=0$, the functional $J_{v,b}^\omega$ was introduced by Li--Shi \cite{LS16}, which we will denote by $J_v^\omega$ for simplicity. Furthermore, when $v=0$, we shall abbreviate $J_0^\omega$ as $J^\omega$.

\begin{prop} \label{strict convexity of the modified J-functional}
For $b \geq 0$, the functional $J_{v,b}^\omega$ is strictly convex along any $C^{1,1}$ geodesics in $\cH(X,\hat{\chi})^T$.
\end{prop}
\begin{proof}
Indeed, we can decompose the functional $J_{v,b}^\omega$ as
\[
J_{v,b}^\omega=J_v^\omega+b \int_X \phi \omega^n.
\]
The first term $J_v^\omega$ is strictly convex by \cite[Proposition 3.1]{LS16}. For $C^2$ geodesics $\frac{d^2}{d t^2} \phi_t-|\nabla_t \ddt \phi_t|_h^2=0$, we can easily see that
\[
\frac{d^2}{d t^2} \int_X \phi_t \omega^n=\int_X \big|\nabla_t \ddt \phi_t \big|_h^2 \omega^n \geq 0.
\]
Then we may approximate $C^{1,1}$ geodesics by $\e$-geodesics just as in \cite[Proposition 2.1]{Che00}.
\end{proof}

\begin{dfn}
We say that the functional $J_{v,b}^\omega$ is coercive if there exist constants $\d_c>0$, $B_c>0$ such that
\[
J_{v,b}^\omega(\phi) \geq \d_c I(\phi)-B_c
\]
for all $\phi \in \cH(X,\hat{\chi})^T$.
\end{dfn}

\begin{rk} \label{equivalent definitions of coercivity}
We define Aubin's $J$-functional \cite{Aub84} by
\[
J(\phi)=\int_0^1 \int_X \phi(\hat{\chi}^n-\chi_{t \phi}^n)dt.
\]
Then it is well-known that
\[
\frac{1}{n+1} I(\phi) \leq J(\phi) \leq \frac{n}{n+1} I(\phi).
\]
Moreover, by a straightforward computation, one can see that
\[
J^{\hat{\chi}}(\phi)=I(\phi)-J(\phi).
\]
In conclusion, we see that in the definition of coercivity, the functionals $I$, $J$, or $J^{\hat{\chi}}$ may be used interchangeably. Alternatively, one can characterize coercivity equivariantly using the $d_1$-distance on $\cH(X,\hat{\chi})^T$ (for instance, see \cite{DR17}).
\end{rk}

%==============Section 3==========================
\section{The modified $J$-flow} \label{The modified J-flow}
Let $X$ be a compact complex manifold, $T \subset \Aut_{\rm red}(X)$ a real torus, and let $\omega \in \a$ and $\hat{\chi} \in \b$ be $T$-invariant K\"ahler forms on $X$. We define the modified $J$-flow as the gradient flow of \eqref{Jvomegab functional}, \ie
\begin{equation} \label{b-modified J-flow}
\ddt \phi_t=-F_b(A_t)+c+\theta_v^X(\chi_t),
\end{equation}
where we set $A_t:=g^{i \bar{k}} h_{j \bar{k}}$ and $\chi_t:=\hat{\chi}+\dd \phi_t$ for simplify the notations. Note that $\phi_t \in \cH(X,\hat{\chi})^T$ as long as it exists. When $b \geq 0$ we have already seen that $J_{v,b}^\omega$ is strictly convex along any $C^{1,1}$ geodesics in $\cH(X,\hat{\chi})^T$ (\cf Lemma \ref{strict convexity of the modified J-functional}). On the other hand it is not the case with $b<0$ so that we should take $\e$ sufficiently small according to the initial data in order to obtain the regularity of the flow. In what follows we study the long-time behavior of \eqref{b-modified J-flow} in two cases $b \geq 0$ and $b=-\e<0$ separately.

\subsection{Long-time existence for $b=-\e<0$} \label{long-time existence for b negative}
Assume $b=-\e<0$. Then by \eqref{constant cX} and \eqref{relation of b and c} the constant $c=c_\e$ is given by
\begin{equation} \label{constant ce}
c_\e=\frac{n \a \cdot \b^{n-1}-\e \a^n}{\b^n}=c_X-\e \frac{\a^n}{\b^n}.
\end{equation}
Fix $\phi_0 \in \cH(X,\hat{\chi})^T$ and take a constant $K_0>0$ so that $Q_\omega(\chi_0)=Q(A_0)<K_0$. We set $\e_5:=\min\{\e_1(n,K_0+3C_\theta), \e_2(n,K_0+3C_\theta), \e_3(n,K_0,C_\theta) \}$ and assume that $\e \in (0,\e_5)$, where the constants $\e_1$, $\e_2$, $\e_3$, $C_\theta$ are given in Proposition \ref{properties for non-positive case}, Proposition \ref{strong convexity}, Lemma \ref{estimate along continuous path} and \eqref{Ctheta} respectively. Then the flow
\begin{equation} \label{modified J-flow when b is non-positive}
\ddt \phi_t=-F_{-\e}(A_t)+c_\e+\theta_v^X(\chi_t)
\end{equation}
is parabolic at $t=0$ from the choice of $\e$, so the short-time existence follows from standard theory. Let $T_{\max}$ be the maximal existence time of $\phi_t \in \cH(X,\hat{\chi})^T$. We introduce the operator $\square_{-\e,t}$ by
\[
\square_{-\e,t}:=\ddt+\p_{ij} F_{-\e}(A_t) g^{i \bar{k}} \nabla_j \nabla_{\bar{k}}-\Re(v),
\]
where $\nabla$ denotes the Chern connection with respect to $\omega$.

\begin{lem} \label{ddphi and phi}
Along the flow $\phi_t$ for $t \in [0,T_{\max})$ we have the following:
\begin{enumerate}
\item $|\ddt \phi_t|<C$.
\item $|\phi_t| \leq C(t+1)$.
\item $Q(A_t)<K_0+3C_\theta$.
\end{enumerate}
\end{lem}
\begin{proof}
Differentiating the equation \eqref{modified J-flow when b is non-positive} we have
\[
\square_{-\e,t} \ddt \phi_t=0
\]
for all $t \in [0,T_{\max})$. So the maximum principle shows that
\[
\inf_X \ddt \phi_0 \leq \ddt \phi_t \leq \sup_X \ddt \phi_0
\]
for all $t \in [0,T_{\max})$ which shows (1), and (2) by integrating in $t$. Moreover, the lower bound of $\ddt \phi_t$ implies that
\[
F_{-\e}(A_t) \leq \theta_v^X(\chi_t)+\sup_X\big(F_{-\e}(A_0)-\theta_v^X(\chi_0) \big)<K_0+2C_\theta.
\]
Then it follows from Lemma \ref{estimate along continuous path} that $Q(A_t)<K_0+3C_\theta$ for all $t \in [0,T_{\max})$, and hence we have (3).
\end{proof}

\begin{lem} \label{Laplacian estimate when b is negative}
The modified $J$-flow admits a solution $\phi_t$ for all $t \in [0,\infty)$.
\end{lem}
\begin{proof}
It follows from Lemma \ref{ddphi and phi} that the function $F_{-\e}$ is convex as long as the flow exists. Therefore, by the Evans--Krylov estimate \cite{Kry76,Kry82}, the Schauder estimate, and the standard bootstrapping argument, the $C^\infty$-estimate for $\phi_t$ reduces to a bound on $|\dd \phi_t|$. We take normal coordinates of $\omega$ at a point so that $\chi_t$ is diagonal with eigenvalues $\lambda_1,\ldots,\lambda_n$. By using \eqref{derivative formulas} we compute
\[
\begin{aligned}
\nabla_i \nabla_{\bar{i}} F_{-\e}(A_t)&=\nabla_{\bar{i}} \big( \p_{pq} F_{-\e}(A_t) \cdot \nabla_i A_q^p \big) \\
&=\p_{pq} F_{-\e}(A_t) \cdot \nabla_{\bar{i}} \nabla_i A_q^p+\p_{pq} \p_{rs} F_{-\e}(A_t) \cdot \nabla_{\bar{i}} A_s^r \cdot \nabla_i A_q^p \\
&=\p_{pq} F_{-\e}(A_t) \cdot \nabla_{\bar{i}} \nabla_i h_{q \bar{p}}+\p_{pq} \p_{rs} F_{-\e}(A_t) \cdot \nabla_{\bar{i}} h_{s \bar{r}} \cdot \nabla_i h_{q \bar{p}} \\
&=\p_{pp} F_{-\e}(A_t) \cdot (\nabla_p \nabla_{\bar{p}} h_{i \bar{i}}+\Rm \ast h)+\p_{pq} \p_{rs} F_{-\e}(A_t) \cdot \nabla_{\bar{i}} h_{s \bar{r}} \cdot \nabla_i h_{q \bar{p}},
\end{aligned}
\]
where $\Rm$ denotes the Riemannian curvature of $\omega$. By the definition of Hamiltonians $\theta_v^X(\chi_t)$ we have
\[
\begin{aligned}
\nabla_i \nabla_{\bar{i}} \theta_v^X(\chi_t)&=\nabla_i (v^k h_{k \bar{i}}) \\
&=\nabla_i v^k \cdot h_{k \bar{i}}+v^k \nabla_i h_{k \bar{i}} \\
&=\nabla_i v^k \cdot h_{k \bar{i}}+v^k \nabla_i \hat{h}_{k \bar{i}}+v^k \nabla_i \phi_{k \bar{i}} \\
&=\nabla_i v^k \cdot h_{k \bar{i}}+v^k \nabla_i \hat{h}_{k \bar{i}}+v^k \nabla_k \phi_{i \bar{i}} \\
&\leq C_1 |h|+v(\Delta_g \phi)+C_1,
\end{aligned}
\]
where $\hat{h}$ denotes the Riemannian metric corresponding to $\hat{\chi}$, and we used $h_{k\bar{i}}=\hat{h}_{k \bar{i}}+\phi_{k \bar{i}}$. By \eqref{modified J-flow when b is non-positive} we obtain
\[
\begin{aligned}
&\ddt \log \Tr_\omega \chi_t \\
&=\frac{1}{\Tr_\omega \chi_t} \nabla_i \nabla_{\bar{i}} \big( -F_{-\e}(A_t)+\theta_v^X(\chi_t) \big) \\
&=\frac{1}{\Tr_\omega \chi_t} \big( -\p_{pp} F_{-\e}(A_t) \cdot \nabla_p \nabla_{\bar{p}} h_{i \bar{i}}-\p_{pp} F_{-\e}(A_t) \cdot \Rm \ast h-\p_{pq} \p_{rs} F_{-\e}(A_t) \cdot \nabla_{\bar{i}} h_{s \bar{r}} \cdot \nabla_i h_{q \bar{p}} \\
&+\nabla_i \nabla_{\bar{i}} \theta_v^X(\chi_t) \big) \\
&\leq \frac{1}{\Tr_\omega \chi_t} \big( -\p_{pp} F_{-\e}(A_t) \cdot \nabla_p \nabla_{\bar{p}} h_{i \bar{i}}+|\p_{pp} F_{-\e}(A_t)| |\Rm| |h|-\p_{pq} \p_{rs} F_{-\e}(A_t) \cdot \nabla_{\bar{i}} h_{s \bar{r}} \cdot \nabla_i h_{q \bar{p}} \\
&+C_1 |h|+v(\Delta_g \phi)+C_1 \big).
\end{aligned}
\]
Also a direct computation shows that
\[
\nabla_p \nabla_{\bar{p}} \log \Tr_\omega \chi_t=\frac{\nabla_p \nabla_{\bar{p}} \Tr_\omega \chi_t}{\Tr_\omega \chi_t}-\frac{|\nabla_p \Tr_\omega \chi_t|^2}{(\Tr_\omega \chi_t)^2}=\frac{\nabla_p \nabla_{\bar{p}} h_{i\bar{i}}}{\Tr_\omega \chi_t}-\frac{|\nabla_i h_{p \bar{i}}|^2}{(\Tr_\omega \chi_t)^2},
\]
\[
v(\log \Tr_\omega \chi_t)=\frac{v(\Tr_\omega \chi_t)}{\Tr_\omega \chi_t}
=\frac{1}{\Tr_\omega \chi_t} \big( v(\Tr_\omega \hat{\chi})+v(\Delta_g \phi) \big) \geq \frac{1}{\Tr_\omega \chi_t} \big( v(\Delta_g \phi)-C_2 \big).
\]
By using the strong convexity (Proposition \ref{strong convexity}) we observe that
\[
\begin{aligned}
&\p_{pq} \p_{rs} F_{-\e}(A_t) \cdot \nabla_{\bar{i}} h_{s \bar{r}} \cdot \nabla_i h_{q \bar{p}}+\p_{pp} F_{-\e}(A_t) \frac{|\nabla_i h_{p \bar{i}}|^2}{\Tr_\omega \chi_t} \\
&\geq \p_{pq} \p_{rs} F_{-\e}(A_t) \cdot \nabla_{\bar{i}} h_{s \bar{r}} \cdot \nabla_i h_{q \bar{p}}+\p_{pp} F_{-\e}(A_t) \frac{|\nabla_i h_{p \bar{j}}|^2}{\lambda_j} \\
&\geq 0.
\end{aligned}
\]
Combining all things together we have
\[
\square_{-\e,t} \log \Tr_\omega \chi_t \leq \frac{1}{\Tr_\omega \chi_t} \big( |\p_{pp} F_{-\e}(A_t)| |\Rm| |h|+C_1 |h|+C_3 \big),
\]
where $\Tr_\omega \chi_t=\sum_i \lambda_i$, $|h|=(\sum_i \lambda_i^2)^{1/2}$ and
\[
0<-\p_{pp} F_{-\e}(A_t)=-\p_p f_{-\e}(\lambda)=\frac{1}{\lambda_p^2}-\frac{\e}{\lambda_pS_n(\lambda)}.
\]
Since $\lambda_i$ is uniformly bounded from below this shows that the right hand side is bounded from above by some uniform constant $C_4>0$. Therefore if we set $G:=\log \Tr_\omega \chi_t-2C_4t$ then for any $T \in (0,T_{\max})$ and maximum point $(x_0,t_0)$ of $G$ on $X \times [0,T]$ with $t_0>0$ we must have
\[
0 \leq \square_{-\e,t} G \leq -C_4,
\]
but this is impossible. It follows that
\[
\Tr_\omega \chi_t \leq e^{2C_4t+C_5}
\]
on $[0,T_{\max})$ which gives the uniform upper bound for $\lambda_i$.
\end{proof}

\subsection{Convergence for $b \geq 0$}
In this subsection we will give a proof of Theorem \ref{existence of solutions to the generalized equation} by showing that the modified $J$-flow
\begin{equation} \label{modified J-flow when b is non-negative}
\ddt \phi_t=-F_b(A_t)+c+\theta_v^X(\chi_t)
\end{equation}
converges smoothly to the solution \eqref{generalized modified J-equation}, under the assumption that there exists a subsolution $\hat{\chi} \in \b$ satisfying
\begin{equation} \label{subsolution as a reference Kahler form}
P_\omega(\hat{\chi})<c+\theta_v^X(\hat{\chi}),
\end{equation}
which is chosen as a reference K\"ahler form. Define the operator
\[
\square_{b,t}:=\ddt+\p_{ij} F_b(A_t)g^{i \bar{k}} \nabla_j \nabla_{\bar{k}}-\Re(v),
\]
and let $T_{\max}$ be the maximal existence time of the flow.

\begin{lem}
There exists a uniform constant $C>0$ such that
\[
\bigg| \ddt \phi_t \bigg|+Q(A_t)<C.
\]
In particular, the eigenvalues of $A_t$ have a uniform positive lower bound.
\end{lem}
\begin{proof}
The proof is similar to that of Lemma \ref{ddphi and phi}.
\end{proof}

\begin{lem} \label{Laplacian estimate}
There exists a uniform constant $N>0$, $C>0$ such that
\begin{equation} \label{trace is bounded by osc}
\Tr_\omega \chi_t \leq C e^{N(\phi_t-\inf_{M \times [0,T]} \phi_t)}
\end{equation}
for all $t \in [0,T]$ and $T \in [0,T_{\max})$.
\end{lem}
\begin{proof}
From the assumption we can choose a constant $\epsilon>0$ sufficiently small so that
\[
P_\omega(\hat{\chi})<c+\theta_v^X(\hat{\chi})-2 \epsilon.
\]
We choose normal coordinates for $\omega$ so that $\chi_t$ is diagonal with entries $\lambda_1 \leq \ldots \leq \lambda_n$. The metric $\hat{\chi}$ may not be diagonal, but we denote its eigenvalues by $\mu_1,\ldots,\mu_n$. As for the first part, the calculation in Lemma \ref{Laplacian estimate when b is negative} can be applied without modification (indeed, one can easily check that the constants $C_1,\ldots,C_4$ in the proof of Lemma \ref{Laplacian estimate when b is negative} can all be chosen independently of $t$). Consequently, we obtain
\[
\square_{b,t} \log \Tr_\omega \chi_t \leq C_1.
\]
The following discussion is restricted to the case $b \geq 0$, and our goal is to obtain a $t$-independent upper bound for the eigenvalues $\lambda_i$. We first compute
\[
\square_{b,t} \phi_t=-F_b(A_t)+c+\theta_v^X(\hat{\chi})+\p_{ii} F_b(A_t) (h_{ii}-\hat{h}_{i \bar{i}}).
\]
Hence if we set $G:=\log \Tr_\omega \chi_t-N \phi_t$ then
\[
\square_{b,t} G \leq N \epsilon+N F_b(A_t)-Nc-N \theta_v^X(\hat{\chi})+N \p_{ii} F_b(A_t)(\hat{h}_{i \bar{i}}-h_{i \bar{i}}),
\]
where the large constant $N>0$ is taken so that $C_1 \leq N \epsilon$. For any fixed $T \in (0,T_{\max})$ the maximum principle shows that
\[
-\epsilon \leq F_b(A_t)-c-\theta_v^X(\hat{\chi})+\p_{ii} F_b(A_t)(\hat{h}_{i \bar{i}}-h_{i \bar{i}})
\]
at the maximum point $(x_0,t_0)$ of $G$ on $X \times [0,T]$ if $t_0>0$. In particular,
\[
P_\omega(\hat{\chi})+\epsilon \leq F_b(A_t)+\p_{ii} F_b(A_t)(\hat{h}_{i \bar{i}}-h_{i \bar{i}})
\]
from our choice of $\epsilon$, which can be written as
\[
\max_k \sum_{i \neq k} \frac{1}{\mu_i}+\epsilon \leq f_b(\lambda)+\p_i f_b(\lambda) \cdot (\hat{h}_{i \bar{i}}-\lambda_i).
\]
Following \cite[Section 2]{CS17} we define the function
\[
\tilde{f}_b(\lambda_1,\ldots,\lambda_{n-1}):=\lim_{\lambda_n \to \infty} f_b(\lambda_1,\ldots,\lambda_n)=\sum_{i=1}^{n-1} \frac{1}{\lambda_i}.
\]
By using the uniform positive lower bound for $\lambda_i$, one can easily observe that for any constant $\tau>0$ there exists $K=K(\tau)>0$ such that if $\lambda_n \geq K$ then
\[
f_b (\lambda_1,\ldots,\lambda_n) \leq \tilde{f}_b (\lambda_1,\ldots,\lambda_{n-1})+\tau,
\]
\[
\sum_{i=1}^{n-1} \big( \p_i f_b (\lambda_1,\ldots,\lambda_n)-\p_i \tilde{f}_b (\lambda_1,\ldots,\lambda_{n-1}) \big) \cdot (\hat{h}_{i \bar{i}}-\lambda_i)=\sum_{i=1}^{n-1} \frac{b}{\lambda_1 \cdots \lambda_n} \bigg(1-\frac{\hat{h}_{i \bar{i}}}{\lambda_i} \bigg) \leq \tau,
\]
\[
\p_n f_b(\lambda) \cdot (\hat{h}_{n \bar{n}}-\lambda_n)=\frac{1}{\lambda_n} \bigg(1+\frac{b}{\lambda_1 \cdots \lambda_{n-1}} \bigg) \bigg( 1-\frac{\hat{h}_{n \bar{n}}}{\lambda_n} \bigg) \leq \tau.
\]
Combining with the convexity of $\tilde{f}_b$ yields that
\[
\begin{aligned}
&f_b(\lambda)+\p_i f_b(\lambda) \cdot (\hat{h}_{i \bar{i}}-\lambda_i) \\
&\leq \tilde{f}_b(\lambda_1,\ldots,\lambda_{n-1})+\sum_{i=1}^{n-1} \p_i \tilde{f}_b(\lambda_1,\ldots,\lambda_{n-1}) \cdot (\hat{h}_{i \bar{i}}-\lambda_i)+3 \tau \\
&\leq \tilde{f}_b(\hat{h}_{1 \bar{1}},\ldots,\hat{h}_{n-1 \overline{n-1}})+3 \tau \\
&\leq \max_k \sum_{i \neq k} \frac{1}{\mu_i}+3\tau,
\end{aligned}
\]
where in the last inequality we used the fact that diagonal entries $(\hat{h}_{1 \bar{1}},\ldots,\hat{h}_{n \bar{n}})$ lies in the convex full of the permutations of the $\mu_i$ by the Schur--Horn theorem. Eventually if we choose $\tau=\epsilon/6$ then we have $\lambda_n \leq K(\tau)$, which yields the desired estimate.
\end{proof}
Once we obtain \eqref{trace is bounded by osc} then we can show the uniform $C^0$-estimate for $\phi_t$ by using exactly the same argument as \cite{SW08,Wei04,Wei06} since it does not depend on the equation \eqref{modified J-flow when b is non-negative}. The higher order estimate follows from the Evans--Krylov estimate \cite{Kry76,Kry82}, the Schauder estimate and the standard bootstrapping argument. Consequently, all derivatives of $\phi_t$ are uniformly bounded independent of $t$, and in particular, the flow $\phi_t$ exists for all $t \in [0,\infty)$. Now we establish the following:

\begin{thm} \label{convergence of the modified J-flow}
Assume that the reference K\"ahler form $\hat{\chi} \in \b$ is a subsolution \eqref{subsolution as a reference Kahler form}. Then along the modified $J$-flow \eqref{modified J-flow when b is non-negative}, the K\"ahler form $\chi_{\phi_t}$ converges smoothly to a solution of \eqref{generalized modified J-equation} as $t \to \infty$.
\end{thm}

\begin{proof}
For $\phi \in \cH(X,\hat{\chi})^T$ we define
\[
E_{v,b}^\omega(\phi):=\int_X \big( F_{\omega,b}(\chi_\phi)-c-\theta_v^X(\chi_\phi) \big)^2 \chi_\phi^n.
\]
We will compute the derivative of $E_{v,b}^\omega$ along the flow. To simplify the notations we set
\[
\s:=\ddt \phi_t=-F_b(A_t)+c+\theta_v^X(\chi_t), \quad \nu:=\frac{\omega^n}{\chi_t^n}.
\]
In normal coordinates of $\chi_t$ we compute
\[
\ddt \s=h^{i \bar{\ell}} h^{k \bar{j}} g_{i \bar{j}} \s_{k \bar{\ell}}+b \nu \Delta_t \s+v(\s).
\]
Hence we have
\[
\begin{aligned}
\ddt E_{v,b}^\omega&=2 \int_X \s \ddt \s \chi_t^n+\int_X \s^2 \Delta_t \s \chi_t^n \\
&=2 \int_X \bigg( \s \ddt \s-\s |\nabla \s|_h^2 \bigg) \chi_t^n \\
&=2 \int_X \big( h^{i \bar{\ell}} h^{k \bar{j}} g_{i \bar{j}} \s \s_{k \bar{\ell}}+b \nu h^{k \bar{\ell}} \s \s_{k \bar{\ell}}+\s v(\s)-h^{k \bar{\ell}} \s \s_k \s_{\bar{\ell}} \big) \chi_t^n \\
&=-2 \int_X \big( h^{i \bar{\ell}} h^{k \bar{j}} g_{i \bar{j}} \s_k \s_{\bar{\ell}}+h^{i \bar{\ell}} h^{k \bar{j}} g_{i \bar{j},\bar{\ell}} \s \s_k+b \nu_{\bar{\ell}} h^{k \bar{\ell}} \s \s_k+b \nu h^{k \bar{\ell}} \s_k \s_{\bar{\ell}}-\s v(\s)+h^{k \bar{\ell}} \s \s_k \s_{\bar{\ell}} \big) \chi_t^n \\
&=-2 \int_X \big( h^{i \bar{\ell}} h^{k \bar{j}} g_{i \bar{j}} \s_k \s_{\bar{\ell}}+h^{i \bar{\ell}} h^{k \bar{j}} g_{i \bar{\ell},\bar{j}} \s \s_k+b \nu_{\bar{\ell}} h^{k \bar{\ell}} \s \s_k+b \nu h^{k \bar{\ell}} \s_k \s_{\bar{\ell}}-\s v(\s)+h^{k \bar{\ell}} \s \s_k \s_{\bar{\ell}} \big) \chi_t^n \\
&=-2 \int_X \big( h^{i \bar{\ell}} h^{k \bar{j}} g_{i \bar{j}} \s_k \s_{\bar{\ell}}+h^{k \bar{j}} (\Tr_{\chi_t} \omega)_{\bar{j}} \s \s_k+b \nu_{\bar{\ell}} h^{k \bar{\ell}} \s \s_k+b \nu h^{k \bar{\ell}} \s_k \s_{\bar{\ell}}-\s v(\s) \\
&-h^{k \bar{\ell}} \s \s_k (\Tr_{\chi_t} \omega+b \nu-\theta_v^X(\chi_t))_{\bar{\ell}} \big) \chi_t^n \\
&=-2 \bigg( \int_X |\nabla \s|_\eta^2 \chi_t^n+b \int_X |\nabla \s|_h^2 \omega^n \bigg) \\
& \leq 0,
\end{aligned}
\]
where the Hermitian metric $\eta$ is defined by $\eta_{i \bar{j}}:=g^{k \bar{\ell}} h_{i \bar{\ell}} h_{k \bar{j}}$. Since we have already obtained a uniform $C^\infty$-estimate for $\phi_t$, we know that
\[
\int_0^\infty \bigg( \int_X |\nabla \s|_\eta^2 \chi_t^n+b \int_X |\nabla \s|_h^2 \omega^n \bigg) dt<+\infty.
\]
By passing to a subsequence $t_k$ we get
\[
\int_X |\nabla \s(t_k)|_\eta^2 \chi_{t_k}^n+b \int_X |\nabla \s(t_k)|_h^2 \omega^n \to 0,
\]
and hence $\s(t_k) \to 0$ by $\int_X \s(t_k) \chi_{t_k}^n=0$. This shows that the limit $\chi_{t_\infty}$ satisfies the generalized equation
\[
\Tr_{\chi_{t_\infty}} \omega+b \frac{\omega^n}{\chi_{t_\infty}^n}=c+\theta_v^X(\chi_{t_\infty}),
\]
where the K\"ahler form $\chi_{t_\infty}$ does not depend on the choice of subsequences, thanks to the strict convexity property of $J_{v,b}^\omega$ (\cf Proposition \ref{strict convexity of the modified J-functional}). Thus we conclude that $\chi_{\phi_t}$ converges smoothly to the unique solution of \eqref{generalized modified J-equation}.
\end{proof}

In particular, Theorem \ref{convergence of the modified J-flow} provides a proof of Theorem \ref{existence of solutions to the generalized equation}.

%==============Section 4==========================
\section{Coercivity of the modified $J$-functional} \label{Coercivity of the modified J-functional}
\subsection{Coercivity implies existence}
In this section we will prove Theorem \ref{existence and coercivity}. As in \cite[Proposition 21]{CS17} a first crucial observation is that if $J_v^\omega$ is coercive then so is $J_{v,-\e}^\omega$ for sufficiently small $\e>0$. For a fixed $\omega \in \a$ we take a reference K\"ahler form $\hat{\chi} \in \b$ as a solution to the complex Monge--Amp\`ere equation
\[
\hat{\chi}^n=\frac{\b^n}{\a^n} \omega^n,
\]
which is known to exist by Yau's theorem \cite{Yau78}.

\begin{prop} \label{perturbation of coercivity}
Suppose that there exist uniform constants $\d_c>0$, $B_c>0$ such that
\[
J_v^\omega(\phi) \geq \d_c I(\phi)-B_c
\]
for all $\phi \in \cH(X,\hat{\chi})^T$. Then there exists $\e_6=\e_6(n,\a,\b,\d_c)>0$ such that for all $\e \in (0,\e_6)$ and $\phi \in \cH(X,\hat{\chi})^T$ we have
\[
J_{v,-\e}^\omega(\phi) \geq \frac{\d_c}{2} I(\phi)-B_c.
\]
\end{prop}
\begin{proof}
If we subtract $J_v^\omega$ from $J_{v,-\e}^\omega$ the integrands involving Hamiltonians cancel each other out. We can therefore proceed as in \cite[Proposition 21]{CS17} to see that
\[
\begin{aligned}
J_{v,-\e}^\omega(\phi)&=J_v^\omega(\phi)+\e \int_0^1 \int_X \phi \bigg( \frac{\a^n}{\b^n} \chi_{t \phi}^n-\omega^n \bigg)dt \\
&=J_v^\omega(\phi)+\e \frac{\a^n}{\b^n} \int_0^1 \int_X \phi (\chi_{t \phi}^n-\hat{\chi}^n) dt \\
&=J_v^\omega(\phi)+\e \frac{\a^n}{\b^n} \int_0^1 \int_X t \phi \dd \phi \wedge (\chi_{t \phi}^{n-1}+\chi_{t \phi}^{n-2} \wedge \hat{\chi}+\cdots+\hat{\chi}^{n-1})dt \\
&\geq J_v^\omega(\phi)-\e \frac{\i}{2 \pi} \frac{\a^n}{\b^n} \int_0^1 \int_X \p \phi \wedge \bp \phi \wedge (\chi_{t \phi}^{n-1}+\chi_{t \phi}^{n-2} \wedge \hat{\chi}+\cdots+\hat{\chi}^{n-1})dt.
\end{aligned}
\]
By using the coercivity estimate $J_v^\omega(\phi) \geq \d_c I(\phi)-B_c$ we have
\begin{equation} \label{coercivity estimate}
\begin{aligned}
J_{v,-\e}^\omega(\phi) &\geq \d_c \frac{\i}{2 \pi} \int_X  \p \phi \wedge \bp \phi \wedge (\chi_\phi^{n-1}+\chi_{\phi}^{n-2} \wedge \hat{\chi}+\cdots+\chi^{n-1}) \\
&-\e \frac{\i}{2 \pi} \frac{\a^n}{\b^n} \int_0^1 \int_X \p \phi \wedge \bp \phi \wedge (\chi_{t \phi}^{n-1}+\chi_{t \phi}^{n-2} \wedge \hat{\chi}+\cdots+\hat{\chi}^{n-1})dt-B_c.
\end{aligned}
\end{equation}
Since $\chi_{t \phi}=(1-t)\chi+t \chi_\phi$ the integrands involving $\chi_{t \phi}$ can be expressed as
\[
\chi_{t \phi}^{n-1-k} \wedge \hat{\chi}^k=\sum_{i=0}^{n-1-k} p_i(t) \chi_\phi^{n-1-k-i} \wedge \hat{\chi}^{k+i}
\]
for all $k=0,1,\ldots,n-1$, where $p_i(t):= \binom{n-1-k}{i}(1-t)^i t^{n-1-k-i}$. This shows that the second term in the right hand side of \eqref{coercivity estimate} can be absorbed in the first term if $\e$ is sufficiently small. So there exists $\e_6>0$ such that if $\e \in (0,\e_6)$ then
\[
J_{v,-\e}^\omega(\phi) \geq \frac{\d_c}{2}I(\phi)-B_c
\]
for all $\phi \in \cH(X,\hat{\chi})^T$.
\end{proof}

From now on, we assume that $J_v^\omega(\phi) \geq \d_c I(\phi)-B_c$ holds for all $\phi \in \cH(X,\hat{\chi})^T$. Take a large number $K_0>0$ so that $Q_\omega(\hat{\chi})<K_0$ and set $\e_7:=\min\{\e_4(n,K_0+3 C_\theta), \e_5,\e_6(n,\a,\b,\d_c) \}$, where $\e_4$ is defined in Lemma \ref{F to P} and $\e_5$ is defined in Section \ref{long-time existence for b negative}. Then by Lemma \ref{ddphi and phi} and Lemma \ref{Laplacian estimate when b is negative} the modified $J$-flow starting from $0 \in \cH(X,\hat{\chi})^T$ exists for all $t \in [0,\infty)$ with a uniform bound
\begin{equation} \label{uniform bound for Q}
Q_\omega(\chi_{\phi_t})<K_0+3C_\theta.
\end{equation}
By the coercivity estimate, there is a sequence $t_k \to \infty$ such that
\[
\lim_{k \to \infty} \ddt J_{v,-\e}^\omega(\phi_{t_k})=0,
\]
and hence
\[
\lim_{k \to \infty} \int_X \big(F_{\omega,-\e}(\chi_{\phi_{t_k}})-c_\e-\theta_v^X(\chi_{\phi_{t_k}}) \big)^2 \omega^n=0,
\]
where we used the uniform lower bound for eigenvalues along the flow $\phi_t$. We normalize $\phi_{t_k}$ so that
\[
u_k:=\phi_{t_k}-\sup_X \phi_{t_k}.
\]
Then
\[
\lim_{k \to \infty} \int_X \big(F_{\omega,-\e}(\chi_{u_k})-c_\e-\theta_v^X(\chi_{u_k}) \big)^2 \omega^n=0.
\]
This shows that $F_{\omega,-\e}(\chi_{u_k})-\Re(v)(u_k) \to c_\e+\theta_v^X(\hat{\chi})$ in $L^2$ when $k \to \infty$ with $\e$ fixed. On the other hand, by using the coercivity and monotonicity along the flow, we observe that
\begin{equation} \label{uniform bound for I}
I(u_k) \leq \frac{J_{v,-\e}^\omega(u_k)+B_c}{\d_c/2} \leq \frac{J_{v,-\e}^\omega(0)+B_c}{\d_c/2} \leq C 
\end{equation}
for some uniform constant $C>0$ (independent of $\e$ and $k$). Hence, by choosing a subsequence, we may assume that $u_k \to u_\infty \in \cE^1(X,\hat{\chi})^T$ in $L^1$ and $\chi_{u_k} \to \chi_{u_\infty}:=\hat{\chi}+\dd u_\infty$ weakly as $k \to \infty$, where $\cE^1(X,\hat{\chi})^T$ denotes the space of $T$-invariant $\hat{\chi}$-PSH functions with finite energy (see Remark \ref{equivalent definitions of coercivity} and \cite[Lemma 3.3]{BBGZ13}). So by \cite[Corollary 1.8]{GZ07} the function $u_\infty$ has zero Lelong numbers.

\subsubsection{Local smoothing} \label{Local smoothing} The K\"ahler current $\chi_{u_\infty}$ can be regarded as a solution to \eqref{modified J-equation} in a weak sense. Now we establish estimates for local smoothing of $\chi_{u_\infty}$. Following the general strategy in \cite{Che21,CS17}, we locally approximate the K\"ahler form $\omega$ and the holomorphic vector field $v$ by constant-coefficient ones in each coordinate chart. This approximation makes the smoothing operator commute with differentiation. For the Hessian term, the desired estimate then follows from the monotonicity and convexity of the function $P$. However, the argument for the Hamiltonian term is more delicate. Even if $v$ is locally approximated by a constant-coefficient vector field, controlling the resulting error term requires a uniform gradient bound $|\nabla u_k|_g \leq C$. Obtaining this bound (without assuming the existence of a subsolution) seems to be a highly challenging problem in its own right. While this error term could be eliminated using the rectification theorem, this approach can not be applied directly since $v$ must have a zero on $X$. To address this issue, we perform the annulus trick as follows: let us consider a product manifold $\cX:=X \times \C$. Let $\omega_\C:=\dd |\tau|^2$ be the standard K\"ahler form on $\C$ and $w:=\i \tau \frac{\p}{\p \tau}$ a holomorphic vector field on $\C$ whose real part generates the action of the real torus $T' \simeq \S^1$ on $\C$ by rotations. Set
\[
\omega_\cX:=\omega+\omega_\C, \quad \hat{\chi}_\cX:=\hat{\chi}+\e^{-2} \omega_\C,
\]
where each term on the right-hand side is the pullback, via the projection, of a K\"ahler form defined on each component (we will also omit writing the pullback via the projection in later arguments). Similarly, we extend $v$ and $w$ to holomorphic vector fields on $\cX$ by means of the product structure, and obtain the product torus action $T \times T'$ on $\cX$ as well. Define
\[
\cX_k:=X \times \{\tau \in \C| 2^{-k}<|\tau|<2^k \}, \quad k=1,2,3.
\]
Since the holomorphic vector field $v+w$ has no zeros on $\cX_3$, for any fixed $p \in \cX_3$ we can choose local homomorphic coordinates $(w^1,\ldots,w^{n+1})$ centered at $p$ such that  $\frac{\p}{\p w^1}=v+w$ by the rectification theorem. In this coordinates, after adjusting the affine term, the local potential function $\phi$ of $\omega_\cX$ is expressed as $\phi(w)=\overline{w}^T Gw+O(|w|^3)$ for some positive definite Hermitian matrix $G=(G_{i \bar{j}})$. In particular, we have $|v+w|_{g_\cX}^2(p)=G_{1 \bar{1}}$, where $g_\cX$ denotes the Riemannian metric with respect to $\omega_\cX$. Take $B:=G^{1/2}$ so that the local potential function $\phi$ in the coordinates $z=Bw$ becomes $\phi(z)=|z|^2+O(|z|^3)$. The holomorphic vector field $v+w$ still has constant coefficients in $(z^1,\ldots,z^{n+1})$, which we denote by $c(v+w) \in \C^{n+1}$. Also a simple computation shows that $|c(v+w)|=|v+w|_{g_\cX}(p) \leq \sup_{\cX_3} |v+w|_{g_\cX}$.

Taking this into the account, we can choose sufficiently small constants $R>0$, $\s>0$ and a finite covering $\{B_{j,4R}\}_{j \in \cJ}$ of $\cX_2$ satisfying the following properties:
\begin{itemize}
\item Each $B_{j,4R}$ is a coordinate chart of $\cX$ biholomorphic to a Euclidean ball $B_{4R}(0)$ of radius $4R$ centered at the origin in $\C^{n+1}$, and $\bigcup_{j \in \cJ} B_{j,4R} \subset \cX_3$.
\item The family $\{B_{j,R}\}_{j \in \cJ}$ also forms a covering of $\cX_2$, where each $B_{j,R}$ is biholomorphic to the Euclidean ball $B_R(0) \subset \C^{n+1}$.
\item On each $B_{j,4R}$, we take local potential functions $\phi_{\omega,j}$, $\phi_{\hat{\chi},j}$ as
\[
\omega_\cX=\dd \phi_{\omega,j}, \quad \hat{\chi}_\cX=\dd \phi_{\hat{\chi},j}.
\]
We may assume that in the coordinates $(z^1,\ldots,z^{n+1})$ of $B_{j,4R}$, the holomorphic vector field $v+w$ has constant coefficients $c(v+w) \in \C^{n+1}$ and $\phi_{\omega,j}(z)=|z|^2+O(|z|^3)$. In particular, since $\p \phi_{\omega,j}(z)=O(|z|)$ and $\p \bp O(|z|^3)=O(|z|)$, we may further assume that
\[
(1-\s) \omega_j \leq \omega_\cX \leq (1+\s) \omega_j,
\]
\[
\big| \phi_{\omega,j}-|z|^2 \big| \leq \frac{1}{10000} R^2, \quad |c(v+w)|+|\nabla \phi_{\omega,j}| \leq C
\]
for some constant $C>0$ depending only on $v$, $w$ and $\omega_\cX$, where $\omega_j$ denotes the K\"ahler form on $B_{j,4R}$ with constant coefficients corresponding to the Euclidean metric $g_j$. Also we normalize $\phi_{\hat{\chi},j}$ so that $\phi_{\hat{\chi},j}(0)=0$.
\end{itemize}

For each $j \in \cJ$ and any locally $L^1$-integrable function $f$ on $B_{j,4R}$, the smoothing of $f$ at scale $r \in (0,R)$ is defined by
\begin{equation} \label{smoothing of functions}
f^{(r)}(z):=\int_{\C^{n+1}} r^{-2n-2} \r \bigg( \frac{|y|}{r} \bigg)f(z-y) d \Vol_{g_j}(y), \quad z \in B_{j,3R},
\end{equation}
where $\r(t)$ is a smooth, non-negative function supported on $[0,1]$, constant on $[0,1/2]$, and normalized so that
\[
\int_{B_1(0)} \r(|y|) d \Vol_{g_j}(y)=1.
\]
Furthermore, if $f$ is smooth, then $f^{(r)} \to f$ in $C^\infty$ as $r \to 0$. So once $\e$, $R$, $\s$ and $\{B_{j,4R}\}_{j \in \cJ}$ are fixed, there exists a constant $r_0 \in (0,R)$ such that for all $r \in (0,r_0)$ and $j \in \cJ$ we have
\begin{equation} \label{choice of r0}
\big| \Re(v+w)(\phi_{\hat{\chi},j}^{(r)}-\phi_{\hat{\chi},j}) \big| \leq \s, \quad \big| \theta_v^X(\hat{\chi})^{(r)}-\theta_v^X(\hat{\chi}) \big| \leq \s
\end{equation}
on $B_{j,3R}$. For each $k$ and $j \in \cJ$, define
\[
u_{k,j}:=\phi_{\hat{\chi},j}+u_k, \quad u_{\infty,j}:=\phi_{\hat{\chi},j}+u_\infty.
\]
By Fubini's theorem, we then have $u_{k,j} \to u_{\infty,j}$ in $L_{\rm loc}^1(B_{j,4R})$ as $k \to \infty$. Also, we have
\[
\hat{\chi}_\cX+\dd u_k=\dd u_{k,j}, \quad \hat{\chi}_\cX+\dd u_\infty=\dd u_{\infty,j},
\]
where the second equality holds in the sense of currents.

The following lemma is crucial to our argument. The key is to estimate $R_{\omega_\cX}$ rather than $P_{\omega_\cX}$. Indeed, a simple computation shows that $P_{\omega_\cX}$ produces an $O(\e)$ correction term, and hence is ineffective. However, in view of Proposition \ref{restriction formulas} (2), it suffices to estimate $R_{\omega_\cX}$ to construct a subsolution on $X$.
\begin{lem} \label{R is bounded by P and e2} For any $\e \in (0,\e_7)$, we have
\[
R_{\omega_\cX} \bigg(\dd u_{k,j} \bigg) \leq F_{\omega,-\e}(\chi_{u_k})+\e^2
\]
on $B_{j,4R}$ for all $j \in \cJ$, where the quantity in the left-hand side is calculated on $\cX$.
\end{lem}
\begin{proof}
Let $\lambda_1 \leq \ldots \leq \lambda_n$ be the eigenvalues of $\chi_{u_k}$ with respect to $\omega$. Then the eigenvalues of $\dd u_{k,j}=\chi_{u_k}+\e^{-2} \omega_\C$ with respect to $\omega_\cX$ are $\lambda_1,\ldots,\lambda_n,\e^{-2}$. If $\lambda_{n-1} \leq \e^{-2}$, then
\[
R_{\omega_\cX} \bigg(\dd u_{k,j} \bigg)=\sum_{i=1}^{n-1} \frac{1}{\lambda_i}=P_\omega(\chi_{u_k}).
\]
Otherwise we have
\[
R_{\omega_\cX} \bigg(\dd u_{k,j} \bigg)=\sum_{i=1}^{n-2} \frac{1}{\lambda_i}+\e^2 \leq P_\omega(\chi_{u_k})+\e^2.
\]
Thus combining with Lemma \ref{F to P} and \eqref{uniform bound for Q} we obtain the desired estimate.
\end{proof}

\begin{lem} \label{estimates for local regularization}
There are constants $\e_8 \in (0,\e_7)$ and $\d=\d(K_0,C_\theta)>0$ such that the following holds: for any $\e \in (0,\e_8)$ there exists constant $\s_0>0$ such that if $\s \in (0,\s_0)$ then
\begin{enumerate}
\item $R_{\omega_\cX} \bigg( \dd u_{\infty,j}^{(r)} \bigg)-\theta_v^X(\hat{\chi})-\Re(v+w)(u_{\infty,j}^{(r)}-\phi_{\hat{\chi},j}) \leq c_X-\d \e$,
\item $Q_{\omega_\cX} \bigg(\dd u_{\infty,j}^{(r)} \bigg) \leq K_0+5C_\theta$
\end{enumerate}
on $B_{j,3R}$ for all $j \in \cJ$ and $r \in (0,r_0)$, where $r_0$ is chosen in \eqref{choice of r0}.
\end{lem}
\begin{proof}
In what follows, let $C$ denote a constant depending only on $K_0$ and $C_\theta$, which may vary from line to line. First, by \eqref{uniform bound for Q} we compute
\[
Q_{\omega_\cX} \bigg(\dd u_{k,j} \bigg)=Q_\omega (\chi_{u_k})+\e^2<K_0+3C_\theta+\e^2.
\]
Thus by Lemma \ref{change of omega} we know that
\begin{equation} \label{omega and constant metric Q}
\bigg| Q_{\omega_j} \bigg(\dd u_{k,j} \bigg)-Q_{\omega_\cX} \bigg( \dd u_{k,j} \bigg) \bigg| \leq C \s,
\end{equation}
\begin{equation} \label{omega and constant metric F}
\bigg| R_{\omega_j} \bigg(\dd u_{k,j} \bigg)-R_{\omega_\cX} \bigg( \dd u_{k,j} \bigg) \bigg| \leq C \s.
\end{equation}
By using the convexity of $Q_{\omega_j}$, \eqref{omega and constant metric Q} and the fact that the coefficients of $\omega_j$ are constants we have
\[
\begin{aligned}
\frac{1}{1+\s} Q_{\omega_\cX} \bigg(\dd u_{k,j}^{(r)} \bigg)(z) &\leq
Q_{\omega_j} \bigg(\dd u_{k,j}^{(r)} \bigg)(z) \\
&\leq \int_{\C^{n+1}} r^{-2n-2} \r \bigg(\frac{|y|}{r} \bigg) Q_{\omega_j} \bigg( \dd u_{k,j} \bigg)(z-y) d \Vol_{g_j}(y) \\
&\leq \int_{\C^{n+1}} r^{-2n-2} \r \bigg(\frac{|y|}{r} \bigg) Q_{\omega_\cX} \bigg( \dd u_{k,j} \bigg)(z-y) d \Vol_{g_j}(y)+C \s \\
& \leq K_0+3C_\theta+\e^2+C \s.
\end{aligned}
\]
So if $\s<\e^2$ and $\e$ is small so that
\[
\e^2 \leq \frac{2C_\theta}{K_0+3C_\theta+2C+2},
\]
then
\[
Q_{\omega_\cX} \bigg(\dd u_{k,j}^{(r)} \bigg) \leq K_0+5C_\theta.
\]
By letting $k \to \infty$ we have
\[
Q_{\omega_\cX} \bigg(\dd u_{\infty,j}^{(r)} \bigg) \leq K_0+5C_\theta,
\]
which shows (2). Next, since $\omega_j$ and $v+w$ have constant coefficients, by using \eqref{omega and constant metric F}, convexity of $R_{\omega_j}$, $\Re(w)(u_k)=0$ and Lemma \ref{R is bounded by P and e2} we have
\[
\begin{aligned}
&R_{\omega_j} \bigg( \dd u_{k,j}^{(r)} \bigg)(z)-\Re(v+w)(u_k^{(r)})(z) \\
& \leq \int_{\C^{n+1}} r^{-2n-2} \r \bigg(\frac{|y|}{r} \bigg) \bigg[ R_{\omega_j} \bigg(\dd u_{k,j} \bigg)(z-y)-\Re(v+w)(u_k)(z-y) \bigg] d \Vol_{g_j}(y) \\
& \leq \int_{\C^{n+1}} r^{-2n-2} \r \bigg(\frac{|y|}{r} \bigg) \bigg[ R_{\omega_\cX} \bigg(\dd u_{k,j} \bigg)(z-y)-\Re(v)(u_k)(z-y) \bigg] d \Vol_{g_j}(y)+C \s \\
&\leq \int_{\C^{n+1}} r^{-2n-2} \r \bigg(\frac{|y|}{r} \bigg) \big[ F_{\omega,-\e} (\chi_{u_k})(z-y)-\Re(v)(u_k)(z-y) \big] d \Vol_{g_j}(y)+C \s+\e^2 \\
\end{aligned}
\]
for all $z \in B_{j,3R}$ and $j \in \cJ$. Since $F_{\omega,-\e}(\chi_{u_k})-\Re(v)(u_k) \to c_\e+\theta_v^X(\hat{\chi})$ in $L_{\rm loc}^2(\cX)$, the first term converges to $c_\e+\big(\theta_V(\hat{\chi}) \big)^{(r)}(z)$ as $k \to \infty$. Thus by letting $k \to \infty$ and using \eqref{choice of r0} we get
\begin{equation}
\begin{aligned} \label{Fbbg}
R_{\omega_j} \bigg( \dd u_{\infty,j}^{(r)} \bigg)-\Re(v+w)(u_\infty^{(r)}) &\leq c_\e+\big(\theta_v^X(\hat{\chi}) \big)^{(r)}+C \s+\e^2\\
&\leq c_\e+\theta_v^X(\hat{\chi})+C \s+\e^2
\end{aligned}
\end{equation}
on $B_{j,3R}$ for all $j \in \cJ$ and $r \in (0,r_0)$. Meanwhile, by using (2) and Lemma \ref{change of omega}, we find that
\begin{equation} \label{eqFF}
\bigg|R_{\omega_j} \bigg(\dd u_{\infty,j}^{(r)} \bigg)-R_{\omega_\cX} \bigg(\dd u_{\infty,j}^{(r)} \bigg) \bigg| \leq C \s.
\end{equation}
Combining \eqref{Fbbg} and \eqref{eqFF}, we obtain
\[
R_{\omega_\cX} \bigg(\dd u_{\infty,j}^{(r)} \bigg)-\theta_v^X(\hat{\chi})-\Re(v+w)(u_\infty^{(r)}) \leq c_\e+C \s+\e^2.
\]
By definition, we have $u_{\infty,j}^{(r)}=(\phi_{\hat{\chi},j}+u_\infty)^{(r)}=\phi_{\hat{\chi},j}^{(r)}+u_\infty^{(r)}$, which implies that
\[
\Re(v+w)(u_{\infty,j}^{(r)}-\phi_{\hat{\chi},j})=\Re(v+w)(u_\infty^{(r)})+\Re(v+w)(\phi_{\hat{\chi},j}^{(r)}-\phi_{\hat{\chi},j}) \geq \Re(v+w)(u_\infty^{(r)})-\s
\]
from our choice of $r \in (0,r_0)$. Consequently, we obtain
\[
R_{\omega_\cX} \bigg(\dd u_{\infty,j}^{(r)} \bigg)-\theta_v^X(\hat{\chi})-\Re(v)(u_{\infty,j}^{(r)}-\phi_{\hat{\chi},j}) \leq c_\e+C \s+\e^2.
\]
Recalling that $c_\e=c_X-\e \frac{\a^n}{\b^n}$ by \eqref{constant ce}, we obtain (1) for some $\d>0$ provided that $\s<\e^2$ and $\e$ is sufficiently small.
\end{proof}

\subsubsection{Gluing argument} We now perform suitable modifications of the functions $u_{\infty,j}^{(r)}$ and glue them together to construct a smooth potential function on $\cX_1$. To this end, as explained in Section \ref{Subsolutions}, we need to verify the gluing conditions. Let $\lambda_1 \leq \cdots \leq \lambda_{n+1}$ be the eigenvalues of $\dd u_{\infty,j}^{(r)}$ with respect to $\omega_\cX$. Then Lemma \ref{estimates for local regularization} (2) yields that
\[
\sum_{i=1}^{n+1} \frac{1}{\lambda_i} \leq K_0+5C_\theta,
\]
which gives a uniform bound $\lambda_i \geq (K_0+5C_\theta)^{-1}$. By using this we observe that the form $\dd u_{\infty,j}^{(r)}-\e^2 \omega_\cX$ is K\"ahler and
\[
\begin{aligned}
R_{\omega_\cX} \bigg( \dd u_{\infty,j}^{(r)}-\e^2 \omega_{\cX} \bigg)&=\sum_{i=1}^{n-1} \frac{1}{\lambda_i-\e^2} \\
&=\sum_{i=1}^{n-1} \frac{1}{\lambda_i}+\sum_{i=1}^{n-1} \frac{\e^2}{(\lambda_i-\e^2) \lambda_i} \\
&\leq \sum_{i=1}^{n-1} \frac{1}{\lambda_i}+C \e^2 \\
&=R_{\omega_{\cX}} \bigg( \dd u_{\infty,j}^{(r)} \bigg)+C \e^2
\end{aligned}
\]
by decreasing $\e_8$ if necessary, where $C=C(K_0,C_\theta)>0$ is a constant. Also we have
\[
\begin{aligned}
\Re(v+w)(u_{\infty,j}^{(r)}-\phi_{\hat{\chi},j}-\e^2 \phi_{\omega,j})&=\Re(v+w)(u_{\infty,j}^{(r)}-\phi_{\hat{\chi},j})-\e^2 \Re(v+w)(\phi_{\omega,j})\\
&\geq \Re(v+w)(u_{\infty,j}^{(r)}-\phi_{\hat{\chi},j})-C \e^2,
\end{aligned}
\]
where we used $|c(v+w)|+|\nabla \phi_{\omega,j}| \leq C$. Thus by decreasing $\e_8$ again we get
\[
R_{\omega_\cX} \bigg( \dd u_{\infty,j}^{(r)}-\e^2 \omega_\cX \bigg)-\theta_v^X(\hat{\chi})-\Re(v+w)(u_{\infty,j}^{(r)}-\phi_{\hat{\chi},j}-\e^2 \phi_{\omega,j}) \leq c_X-\d \e+C \e^2<c_X.
\]
Now we fix $\e \in (0,\e_8)$, $\s \in (0,\s_0)$, $R>0$ and $\{B_{j,4R}\}_{j \in \cJ}$. In particular, we have a uniform bound $|\nabla \phi_{\hat{\chi},j}| \leq S$ on each $B_{j,4R}$ for some $S>0$. Define
\[
\varphi_{j,r}:=u_{\infty,j}^{(r)}-\phi_{\hat{\chi},j}-\e^2 \phi_{\omega,j}
\]
so that
\[
\hat{\chi}_\cX+\dd \varphi_{j,r}=\dd u_{\infty,j}^{(r)}-\e^2 \omega_\cX>0,
\]
\[
R_{\omega_\cX} \bigg( \hat{\chi}_\cX+\dd \varphi_{j,r} \bigg)-\theta_v^X(\hat{\chi})-\Re(v+w)(\varphi_{j,r})<c_X.
\]
For $z \in B_{j,3R}$ and $r \in (0,R)$, let
\[
u_{j,r}(z):=\sup_{B_{j,r}(z)} u_{\infty,j},
\]
where $B_{j,r}(z)$ denotes a Euclidean ball of radius $r$ centered at $z$ in $B_{j,4R}$. For $r \in (0,R/2)$, we define
\[
\nu_{j,u_{\infty,j}}(z,r):=\frac{u_{j,\frac{3}{4}R}(z)-u_{j,r}(z)}{\log \big( \frac{3}{4}R \big)-\log r}, \quad z \in B_{j,3R}.
\]
As $r \to 0$ the quantity $\nu_{j,u_{\infty,j}}(z,r)$ converges decreasingly to the Lelong number of $u_{\infty}$ with respect to a reference metric $\hat{\chi}_\cX$:
\[
\lim_{r \to 0} \nu_{j,u_{\infty,j}}(z,r)=\nu_{u_\infty}(z).
\]
Note that $\nu_{u_\infty}(z)=0$ for all $z \in B_{j,3R}$ since $\chi_{u_\infty}$ (and hence its pullback to $\cX$) has zero Lelong numbers. By adapting the idea from B{\l}ocki--Ko{\l}odziej \cite{BK07}, Chen \cite[Lemma 4.2]{Che21} proved the following lemma, which states that the difference between $u_{j,r}(z)$ and $u_{\infty,j}^{(r)}(z)$ can be controlled in terms of $\nu_{j,u_{\infty,j}}(z,r)$.
\begin{lem}
For any $r \in (0,R/2)$ and $z \in B_{j,3R}$ we have
\begin{enumerate}
\item $0 \leq u_{j,r}(z)-u_{j,\frac{r}{2}}(z) \leq (\log 2) \nu_{j,u_{\infty,j}}(z,r)$.
\item $0 \leq u_{j,r}(z)-u_{\infty,j}^{(r)}(z) \leq \eta \nu_{j,u_{\infty,j}}(z,r)$, where $\eta>0$ is a universal constant depending only on $n$.
\end{enumerate}
\end{lem}

\begin{lem} \label{gluing condition}
There exists $r_1 \in (0,r_0)$ such that for any $r \in (0,r_1)$ and $z \in (B_{j,3R} \backslash B_{j,2R}) \cap B_{i,R}$, we have
\[
\varphi_{j,r}(z) \leq \varphi_{i,r}(z)-\frac{\e^2 R^2}{2}.
\]
\end{lem}
\begin{proof}
The function $\nu_{j,u_{\infty,j}}(z,r)$ is increasing with respect to $r$ (when fixing $z$) and upper semi-continuous with respect to $z$ (when fixing $r$). Moreover, the Lelong number $\nu_{u_\infty}$ is zero everywhere. Thus we can apply the Dini--Cartan lemma to know that for a constant $A>0$ determined larter, after decreasing $r_0$ for any $r \in (0,r_0)$ and $z \in B_{j,3R}$ we have
\[
0 \leq \nu_{j,u_{\infty,j}}(z,r) \leq A^{-1} \e^2.
\]
We compute
\[
\begin{aligned}
u_{\infty,j}^{(r)}(z) &\leq u_{j,r}(z) \\
&\leq u_{j,\frac{r}{2}}(z)+(\log 2) A^{-1} \e^2 \\
&\leq \sup_{B_{j,\frac{r}{2}}(z)}(\phi_{\hat{\chi},j}+u_\infty)+(\log 2) A^{-1} \e^2 \\
&\leq \sup_{B_{j,\frac{r}{2}}(z)}\phi_{\hat{\chi},j}+\sup_{B_{j,\frac{r}{2}}(z)} u_\infty+A^{-1} \e^2.
\end{aligned}
\]
It follows that
\[
\begin{aligned}
\varphi_{j,r}(z)&=u_{\infty,j}^{(r)}(z)-\phi_{\hat{\chi},j}(z)-\e^2 \phi_{\omega,j}(z) \\
&\leq \sup_{B_{j,\frac{r}{2}}(z)} u_\infty+\sup_{B_{j,\frac{r}{2}}(z)}\phi_{\hat{\chi},j}-\phi_{\hat{\chi},j}(z)+A^{-1} \e^2-\e^2 \phi_{\omega,j}(z) \\
& \leq \sup_{B_{j,\frac{r}{2}}(z)} u_\infty+Sr+A^{-1}\e^2-3 \e^2 R^2.
\end{aligned}
\]
On the other hand,
\[
\begin{aligned}
u_{\infty,i}^{(r)}(z) &\geq u_{i,r}(z)-\eta A^{-1} \e^2 \\
&=\sup_{B_{i,r}(z)}(\phi_{\hat{\chi},i}+u_\infty)-\eta A^{-1} \e^2 \\
&\geq \sup_{B_{i,r}(z)}u_\infty+\inf_{B_{i,r}(z)} \phi_{\hat{\chi},i}-\eta A^{-1} \e^2.
\end{aligned}
\]
This shows that
\[
\begin{aligned}
\varphi_{i,r}(z)&=u_{\infty,i}^{(r)}(z)-\phi_{\hat{\chi},i}(z)-\e^2 \phi_{\omega,i}(z) \\
&\geq \sup_{B_{i,r}(z)}u_\infty+\inf_{B_{i,r}(z)} \phi_{\hat{\chi},i}-\phi_{\hat{\chi},i}(z)-\eta A^{-1} \e^2-\e^2 \phi_{\omega,i}(z) \\
&\geq \sup_{B_{i,r}(z)}u_\infty-Sr-\eta A^{-1} \e^2-2\e^2 R^2.
\end{aligned}
\]
The above computation yields that
\[
\varphi_{i,r}(z)-\varphi_{j,r}(z) \geq \sup_{B_{i,r}(z)}u_\infty-\sup_{B_{j,\frac{r}{2}}(z)} u_\infty-2Sr-(\eta+1)A^{-1}\e^2+\e^2 R^2.
\]
From the choice of the covering $\{B_{j,4R}\}_{j \in \cJ}$ one can easily see that $B_{j,\frac{r}{2}}(z) \subset B_{i,r}(z)$. Thus
\[
\varphi_{i,r}(z)-\varphi_{j,r}(z) \geq -2Sr-(\eta+1)A^{-1}\e^2+\e^2 R^2.
\]
By choosing $A=\frac{3(\eta+1)}{R^2}$ we have
\[
\varphi_{i,r}(z)-\varphi_{j,r}(z) \geq -2Sr+\frac{2 \e^2 R^2}{3}.
\]
So if we set $r_1=\frac{\e^2 R^2}{12S}$, then for any $r \in (0,r_1)$, we conclude that
\[
\varphi_{j,r}(z) \leq \varphi_{i,r}(z)-\frac{\e^2 R^2}{2}.
\]
\end{proof}

The above discussion prepares the ground for constructing a subsolution on $X$. We now proceed to the proof.

\begin{proof}[Proof of ``coercivity $\Rightarrow$ existence'' in Theorem \ref{existence and coercivity}]
From Lemma \ref{gluing condition}, we know that for any $r \in (0,r_1)$ the regularized maximum $\varphi$ of $\{(B_{j,3R},\varphi_{j,r})\}_{j \in \cJ}$ (for a sufficiently small vector $\eta$) is smooth on $\cX_1$. Define $\Omega:=\hat{\chi}_\cX+\dd \varphi$. Then by Proposition \ref{gluing of subsolutions}, $\Omega$ is a K\"ahler form on $\cX_1$ which satisfies
\[
R_{\omega_\cX}(\Omega)-\theta_v^X(\hat{\chi})-\Re(v+w)(\varphi)<c_X.
\]
Let
\[
\tilde{\varphi}:=\int_{\tau \in T'}\tau^\ast \varphi dT', \quad \tilde{\Omega}=\hat{\chi}_\cX+\dd \tilde{\varphi}.
\]
Since $\omega_\cX$, $\hat{\chi}_\cX$, $\theta_v^X(\hat{\chi})$ and $\Re(v+w)$ are $T'$-invariant, Proposition \ref{averaging of potentials and forms} and $\Re(w)(\tilde{\varphi})=0$ yield that
\[
R_{\omega_\cX}(\tilde{\Omega})-\theta_v^X(\hat{\chi})-\Re(v)(\tilde{\varphi})<c_X
\]
on $\cX_1$. Finally, we observe that the restriction of $\tilde{\Omega}$ to $X \simeq X \times \{1\} \subset \cX_1$:
\[
\chi:=\tilde{\Omega}|_X=\hat{\chi}+\dd \tilde{\varphi}|_X
\]
satisfies
\[
P_\omega(\chi)-\theta_v^X(\hat{\chi})-\Re(v)(\tilde{\varphi}|_X)<c_X,
\]
where we used Proposition \ref{restriction formulas} (2) and the fact that $\Re(v)$ is tangent to $X$. Hence, by Proposition \ref{averaging of subsolutions}, if we set
\[
\phi:=\int_{\tau \in T} \tau^\ast \tilde{\varphi}|_X dT, \quad \tilde{\chi}:=\hat{\chi}+\dd \phi,
\]
then the K\"ahler form $\tilde{\chi} \in \b$ satisfies
\[
P_\omega(\tilde{\chi})-\theta_v^X(\tilde{\chi})<c_X.
\]
The desired result follows from Theorem \ref{existence of solutions to the generalized equation}.
\end{proof}

\subsection{Existence implies coercivity}
We will prove the converse direction of Theorem \ref{existence and coercivity} and Theorem \ref{solvability is independent of omega}. As in Proposition \ref{perturbation of coercivity}, the terms involving the Hamiltonian cancel out when taking the difference of the functionals, so that, in effect, there is no substantial difference from the usual $J$-equation case \cite[Proposition 22]{CS17}. However, we will give it here for the sake of completeness.
\begin{proof}[Proof of ``existence $\Rightarrow$ coerxivity'' in Theorem \ref{existence and coercivity}]
Assume that there exists a solution $\hat{\chi} \in \b$ to the modified $J$-equation
\[
\Tr_{\hat{\chi}} \omega=c_X+\theta_v^X(\hat{\chi}).
\]
Take a sufficiently small $\d>0$ so that $\omega':=\omega-\d \hat{\chi}>0$. Then
\[
\Tr_{\hat{\chi}} \omega'=c_X-n\d+\theta_v^X(\hat{\chi}),
\]
\ie the K\"ahler form $\hat{\chi}$ solves the modified $J$-equation with respect to $\omega'$. In particular, we have $J_v^{\omega'}(\phi) \geq -C$ for all $\phi \in \cH(X,\hat{\chi})^T$ \cite[Corollary 3.2]{LS16}. For all $\phi \in \cH(X,\hat{\chi})^T$, we compute
\[
J_v^\omega(\phi)=J_v^{\omega'}(\phi)-n\d \int_0^1 \int_X \phi(\chi_{t\phi}^n-\hat{\chi} \wedge \chi_{t\phi}^{n-1})dt.
\]
Integrating by parts we get
\[
\begin{aligned}
-\int_X \phi (\chi_{t \phi}^n-\hat{\chi} \wedge \chi_{t \phi}^{n-1}) &=-\int_X \phi \bigg( t \dd \phi \bigg) \wedge \chi_{t \phi}^{n-1}\\
&=\frac{\i}{2 \pi}t \int_X \p \phi \wedge \bp \phi \wedge \chi_{t \phi}^{n-1} \\
&=\frac{\i}{2 \pi} \sum_{i=1}^{n-1} p_i(t) \int_X \p \phi \wedge \bp \phi \wedge \chi_\phi^i \wedge \hat{\chi}^{n-1-i},
\end{aligned}
\]
where $p_i(t):=\binom{n-1}{i} t^{i+1} (1-t)^{n-1-i}$. We take a constant $\kappa>0$ (depending only on $n$) sufficiently small so that
\[
\int_0^1 p_i(t) dt \geq \kappa, \quad i=1,\ldots n-1.
\]
Then we have
\[
\begin{aligned}
-n\d \int_0^1 \int_X \phi (\chi_{\phi_t}^n-\hat{\chi} \wedge \chi_{\phi_t}^{n-1}) dt&=\frac{\i}{2 \pi} n\d \sum_{i=1}^{n-1} \int_0^1 p_i(t) dt \int_X \p \phi \wedge \bp \phi \wedge \chi_\phi^i \wedge \hat{\chi}^{n-1-i} \\
&\geq \frac{\i}{2 \pi} n\d \kappa \int_X \p \phi \wedge \bp \phi \wedge \sum_{i=1}^{n-1} \chi_\phi^i \wedge \hat{\chi}^{n-1-i} \\
&=n\d \kappa I(\phi).
\end{aligned}
\]
It then follows that
\[
J_v^\omega(\phi) \geq n \d \kappa I(\phi)-C
\]
for all $\phi \in \cH(X,\hat{\chi})^T$.
\end{proof}
As a corollary of this we can show that the solvability of the $J$-equation does not depend on a choice of $\omega \in \a$.
\begin{proof}[Proof of Theorem \ref{solvability is independent of omega}]
Thanks to Theorem \ref{existence and coercivity}, it suffices to show that the coercivity of $J_v^\omega$ does not depend on the choice of $\omega$. If $\omega, \omega' \in \a$ are $T$-invariant K\"ahler forms there exists a $T$-invariant smooth function $\varphi_0$ such that $\omega=\omega'+\dd \varphi_0$. Then we compute
\[
\begin{aligned}
J_v^\omega(\phi)-J_v^{\omega'}(\phi)&=n \int_0^1 \int_X \phi \dd \varphi_0 \wedge \chi_{t \phi}^{n-1} dt \\
&=n \int_0^1 \int_X \varphi_0 \dd \phi \wedge \chi_{t \phi}^{n-1} dt \\
&=\int_0^1 \int_X \varphi_0 \frac{d \chi_{t \phi}^n}{dt} dt \\
&=\int_X \varphi_0(\chi_\phi^n-\hat{\chi}^n),
\end{aligned}
\]
and hence we have
\[
\big|J_v^\omega(\phi)-J_v^{\omega'}(\phi) \big| \leq \int_X |\varphi_0| \chi_\phi^n+C \leq C
\]
as desired.
\end{proof}

%==============Section 6==========================
\section{A Nakai--Moishezon type criterion for smooth projective toric varieties} \label{A Nakai-Moishezon type criterion for smooth projective toric varieties}
\subsection{Local $C^2$-estimate for the generalized equation}
Let $X$ be a compact complex manifold, $T \subset \Aut_{\rm red}(X)$ a real torus, and let $\omega \in \a$ and $\hat{\chi} \in \b$ be $T$-invariant K\"ahler forms. Let $b \geq 0$ and $c>0$ be constants related by \eqref{relation of b and c}. In this subsection we establish the local $C^2$-estimate of the generalized equation for $\chi_\phi=\hat{\chi}+\dd \phi$ ($\phi \in \cH(X,\hat{\chi})^T$):
\begin{equation} \label{generalized modified J-equation r}
\Tr_{\chi_\phi} \omega+b \frac{\omega^n}{\chi_\phi^n}=c+\theta_v^X(\chi_\phi).
\end{equation}

\begin{lem} \label{local estimate for generalized equation}
Let $\phi \in \cH(X,\hat{\chi})^T$ be a solution to \eqref{generalized modified J-equation r}. Assume that there exists a K\"ahler form $\tilde{\chi}=\hat{\chi}+\dd \psi$ defined on the closure $\overline{\Omega}$ of an open subset $\Omega \subset X$ satisfying
\[
P_\omega(\tilde{\chi})-\theta_v^X(\hat{\chi})-\Re(v)(\psi)<c-2\d
\]
for some $\d>0$ (where $\Omega$, $\psi$ and $\tilde{\chi}$ do not have to be $T$-invariant). Then
\[
\Tr_\omega \chi_\phi \leq C e^{N(\phi-\inf_{\overline{\Omega}} \phi)}
\]
on $\overline{\Omega}$, where the constants $C>0$, $N>0$ depend on $\tilde{\chi}$, $\psi$, $\d$, $C_\theta$, an upper bound of $c$ and the maximum of $\Tr_\omega \chi_\phi$ on $\p \Omega$.
\end{lem}

As a concrete situation in which the above Lemma is applied, we consider $X$ to be toric and $\Omega$ an open subset of $X$ containing the union of all toric divisors. This Lemma is required for the inductive proof of Theorem \ref{subsolutions near subvarieties} on toric subvarieties.

\begin{proof}[Proof of Lemma \ref{local estimate for generalized equation}]
The proof is essentially the same as Lemma \ref{Laplacian estimate}. Set $\varphi:=\phi-\psi$ so that $\chi_\phi=\tilde{\chi}+\dd \varphi$. Define the operator $\cL_b$ by
\[
\cL_b:=\p_{ij} F_b(A) g^{i \bar{k}} \nabla_j \nabla_{\bar{k}}-\Re(v),
\]
where $A=g^{i \bar{k}} h_{j \bar{k}}$ and $\nabla$ denotes the Chern connection with respect to $\omega$. We take normal coordinates of $\omega$ on which $\chi_\phi$ is diagonal. From the equation \eqref{generalized modified J-equation r} we know that $\chi_\phi \geq C_1^{-1} \omega$ for some constant $C_1>0$ depending on $C_\theta$ and an upper bound of $c$. Then we have
\[
\cL_b \log \Tr_\omega \chi_\phi \leq C_2
\]
(indeed, we may substitute $\ddt \phi_t=0$ to the computation in Lemma \ref{Laplacian estimate} since $\phi$ is a stationary point of the modified $J$-flow). Similarly by using \eqref{generalized modified J-equation r} and $\varphi=\phi-\psi$ a direct computation shows that
\[
\cL_b \varphi=-F_b(A)+c+\theta_v^X(\hat{\chi})+\Re(v)(\psi)+\p_{ii} F_b(A)(h_{i\bar{i}}-\tilde{h}_{i\bar{i}}),
\]
where $\tilde{h}$ denotes the Riemannian metric corresponding to $\tilde{\chi}$. Define the function $G \colon \overline{\Omega} \to \R$ by $G:=\log \Tr_\omega \chi_\phi-N \varphi$. Then the above computation shows that
\[
\cL_b G \leq C_2+N F_b(A)-Nc-N \theta_v^X(\hat{\chi})-N \Re(v)(\psi)+N \p_{ii} F_b(A) (\tilde{h}_{i \bar{i}}-h_{i \bar{i}}).
\]
We take a large constant $N>0$ so that $C_2 \leq N \d$. If the function $G$ takes maximum at some point $x_0 \in \Omega$ then we have
\[
-\d \leq F_b(A)-c-\theta_v^X(\hat{\chi})-\Re(v)(\psi)+\p_{ii} F_b(A)(\tilde{h}_{i\bar{i}}-h_{i \bar{i}}),
\]
and hence
\[
P_\omega(\tilde{\chi})+\d \leq F_b(A)+\p_{ii} F_b(A)(\tilde{h}_{i\bar{i}}-h_{i \bar{i}})
\]
at $x_0$. Thus we can apply the same argument in the second half of the proof for Lemma \ref{generalized modified J-equation r} just by replacing $\hat{\chi}$ with $\tilde{\chi}$, and get
\[
\Tr_\omega \chi_\phi \leq C e^{N(\varphi-\inf_{\overline{\Omega}} \varphi)}.
\]
Finally, by using a bound of $\psi$, we obtain the desired estimate.
\end{proof}

\subsection{Continuity method and induction for toric subvarieties}
In what follows let $X$ be an $n$-dimensional smooth projective toric variety associated to a fan $\Sigma$, and let $T \subset T^\C \simeq (\C^\ast)^n$ denote the real torus and its complexification associated with $X$. We denote $N:=\Hom(\C^\ast,T^\C)$ the lattice of one-parameter subgroups of $T^\C$. The fan $\Sigma$ is then a collection of strictly convex rational polyhedral cones in $N_\R:=N \otimes_\Z \R$.

Let $v$ be a holomorphic vector field with $\Im(v) \in \ft$, and assume that K\"ahler forms $\omega \in \a$, $\hat{\chi} \in \b$ are $T$-invariant. For any $p$-dimensional toric subvarieties\footnote{In this paper the term ``$p$-dimensional subvariety'' means that it is reduced, but not necessarily irreducible or purely $p$-dimensional.} $Y \subset X$ we define the set $\Gamma_{\omega,\hat{\chi}}(Y)$ consisting of germs of $\hat{\chi}$-K\"ahler potential $\phi$ satisfying
\[
Q_{X,\omega}\bigg( \hat{\chi}+\dd \phi \bigg)-\theta_v^X(\hat{\chi})-\Re(v)(\phi)<c_X
\]
at $Y \subset X$. When $Y$ is irreducible, it is given as the $T^\C$-invariant subvariety associated with a cone $\tau \in \Sigma$
\[
Y=\overline{O(\tau)} \subset X,
\]
which is the closure of the $T^\C$-orbit corresponding to $\tau$. Let $N_\tau$ be the sublattice of $N$ spanned by the points in $N \cap \tau$ and define $N(\tau)$ by the exact sequence
\[
0 \to N_\tau \to N \to N(\tau) \to 0.
\]
Then the fan $\Sigma(\tau)$ corresponding to $Y$ is given by
\[
\Sigma(\tau)=\{ \overline{\s} \subset N(\tau)_\R| \tau \preceq \s \in \Sigma \},
\]
where $\overline{\s}$ denotes the image cone of $\s$ in $N(\tau)_\R:=N(\tau) \otimes_\Z \R$ under the quotient map $N_\R \to N(\tau)_\R$. In particular, the properties of a cone, rationality, smoothness and strict convexity are inherited by the toric subvariety $Y$. Also the real torus associated with $Y$ is then given by $T_Y:=N(\tau) \otimes_\Z \S^1$. Thus tensoring with $\S^1$ to the projection $N \to N(\tau)$, we obtain a surjective homomorphism $T \to T_Y$. The $T$-action on $Y$ is just obtained by composing this morphism with the $T_Y$-action on $Y$ (for the notation and a more detailed explanation, see \cite[Section 3.2]{CLS11}). Define
\[
\cH(Y,\hat{\chi})^{T_Y}:=\big\{ \phi \in C^\infty(Y;\R)^{T_Y} \big| \chi_\phi:=\hat{\chi}+\dd \phi>0 \big\}.
\]
Since $Y$ is invariant under the $T^\C$-action, $v$ is tangent to $Y$. In particular, restricting \eqref{definition and normalization of Hamiltonians} to $Y$, we find that the vector field $v|_Y$ has a Hamiltonian $\theta_v^Y(\hat{\chi})$ with respect to $\chi|_Y$ which differs from $\theta_v^X(\hat{\chi})$ as
\[
\theta_v^Y(\hat{\chi})=\theta_v^X(\hat{\chi})-I_Y,
\]
where
\[
I_Y:=\frac{1}{\b^p} \int_Y \theta_v^X(\hat{\chi}) \hat{\chi}^p,
\]
and we often omit the notation for restriction to $Y$. Since $\theta_v^Y(\chi_\phi)=\theta_v^Y(\hat{\chi})+v(\phi)$ we have
\begin{equation} \label{change of normalization}
\theta_v^Y(\chi_\phi)=\theta_v^X(\hat{\chi})+v(\phi)-I_Y
\end{equation}
for all $\phi \in \cH(Y,\hat{\chi})^{T_Y}$.

\begin{thm} \label{subsolutions near subvarieties}
Let $X$, $T$, $v$, $\omega$, $\hat{\chi}$, $\a$, $\b$ as above. Assume $m_X>0$ and there exists a $T$-invariant K\"ahler form $\chi \in \b$ such that
\[
\int_Y \big( (c_X+\theta_v^X(\chi)) \chi^p-p \omega \wedge \chi^{p-1} \big)>0
\]
for all $p$-dimensional irreducible toric subvarieties $Y \subset X$ ($p=1,\ldots,n-1$). Then for any $p$-dimensional toric subvarieties $Y \subset X$ ($p=0,1,\ldots,n-1$) we have
\[
\Gamma_{\omega,\hat{\chi}}(Y) \neq \emptyset.
\]
\end{thm}

\begin{proof}
We will prove by induction of $p=\dim Y$.

\vspace{2mm}

\noindent
\textbf{Step 0.} First, we note that if $\dim Y=0$, \ie $Y$ is a finite set of points, then the result easily follows from $m_X=c_X+\min_X \theta_v^X(\hat{\chi})>0$, and the facts that $\Re(v)$ vanishes at $Y$ and every K\"ahler class in a sufficiently small neighborhood of a point is trivial.

\vspace{2mm}

\noindent
\textbf{Step 1.} Assume that $p \geq 1$ and the statement is true for all toric subvarieties $Z \subset X$ of $\dim Z \leq p-1$. For any $p$-dimensional irreducible toric subvariety $Y \subset X$ let us consider the following continuity path $\chi_{\phi_t}=\hat{\chi}+\dd \phi_t \in \b|_Y$ ($t \in [0,\infty)$) with $\phi_t \in \cH(Y,\hat{\chi})^{T_Y}$:
\begin{equation} \label{continuity path on Y}
\Tr_{\chi_{\phi_t}} \omega+b_t \frac{\omega^p}{\chi_{\phi_t}^p}=c_t+\theta_v^Y(\chi_{\phi_t}),
\end{equation}
where
\[
c_t:=c_X(t+1)+I_Y \geq c_X+I_Y \geq c_X+\min_X \theta_v^X(\hat{\chi})=m_X>0,
\]
\[
b_0:=\frac{(c_X+I_Y) \b^p-p\a \cdot \b^{p-1}}{\a^p}>0,
\]
\[
b_t:=\frac{c_t \b^p-p \a \cdot \b^{p-1}}{\a^p}=b_0+c_X \frac{\b^p}{\a^p}t>0
\]
from the assumption. We note that
\[
\begin{aligned}
c_t+\min_Y \theta_v^Y(\chi_{\phi_t})&=c_t+\min_Y \theta_v^Y(\hat{\chi}) \\
&\geq c_X+I_Y+\min_Y(\theta_v^X(\hat{\chi})-I_Y) \\
&\geq c_X+\min_X \theta_v^X(\hat{\chi}) \\
&=m_X>0,
\end{aligned}
\]
which will be required to obtain the interior $C^2$-estimates (see Section \ref{Interior C2-estimates}). Set
\[
\cT:=\{t \in [0,\infty)|\text{\eqref{continuity path on Y} admits a solution $\phi_t \in \cH(Y,\hat{\chi})^{T_Y}$}\}.
\]
By Theorem \ref{existence of solutions to the generalized equation} the solvability of \eqref{continuity path on Y} is equivalent to the existence of $\underline{\phi} \in \cH(Y,\hat{\chi})^{T_Y}$ such that
\begin{equation} \label{solvability of the continuity path on Y}
P_\omega(\chi_{\underline{\phi}})-\theta_v^Y(\chi_{\underline{\phi}})<c_t.
\end{equation}
From the uniform bound of $\theta_v^Y(\chi_{\underline{\phi}})$, we know that any $\underline{\phi} \in \cH(Y,\hat{\chi})^{T_Y}$ satisfies \eqref{solvability of the continuity path on Y} for sufficiently large $t$. It then follows that $\cT$ is open, and there exists a large constant $T>0$ such that $[T,\infty) \subset \cT$.

Now we will show that $\cT$ is closed. Let $t_0:=\inf \cT$ and $D=\cup_i D_i$ be the union of all toric divisors of $Y$. From the inductive hypothesis, there exists an open neighborhood $U_D$ of $D$ in $X$, and a $\hat{\chi}$-K\"ahler potential $\phi_D$ on $\overline{U_D}$ such that
\[
Q_{X,\omega} \bigg( \hat{\chi}+\dd \phi_D \bigg)-\theta_v^X(\hat{\chi})-\Re(v)(\phi_D)<c_X-\d
\]
for some $\d>0$. By Proposition \eqref{restriction formulas} (1) and \eqref{change of normalization}, we have
\[
Q_{Y,\omega} \bigg( \hat{\chi}+\dd \phi_D \bigg)-\theta_v^Y(\hat{\chi})-\Re(v)(\phi_D)<c_X+I_Y-\d \leq c_{t_0}-\d
\]
on the closure $\overline{U_Y}$ of $U_Y:=U_D \cap Y$ in $Y$, since $\Re(v)$ is tangent to $Y$. For $t>t_0$ close to $t_0$, we invoke Lemma \ref{local estimate for generalized equation} to reduce the $C^2$-estimates of $\phi_t$ on $\overline{U_Y}$ to estimates on the boundary $\p U_Y$. However, by Proposition \ref{C2 estimate for g} we obtain bounds of the form
\[
\Tr_\omega \chi_{\phi_t} \leq C e^{N(\phi_t-\inf_{Y \backslash U_Y} \phi_t)}
\]
on $Y \backslash U_Y$. Eventually, we obtain the same inequality globally on $Y$. Thus we can use the argument \cite{SW08,Wei04,Wei06}, the Evans--Klyrov theorem and the Schauder estimates to obtain higher order estimates. This shows that $\cT$ is closed, and hence $\cT=[0,\infty)$.

\vspace{2mm}

\noindent
\textbf{Step 2.}
By Step 1 we have obtained $\phi \in \cH(Y,\hat{\chi})^{T_Y}$ satisfying
\[
Q_{Y,\omega}\bigg( \hat{\chi}+\dd \phi \bigg)-\theta_v^X(\hat{\chi})-\Re(v)(\phi) \leq c_X-2\d
\]
as a solution to \eqref{continuity path on Y} for $t=0$ on $Y$, where the existence of a positive constant $\d>0$ follows from $b_0>0$. For a projective toric variety $X$, the rational, complete fan $\Sigma$ and lattice $N$ allow a canonical $T_Y$-equivariant linearization on the line bundle $\cO_Y(D_i)$ \cite[Section 6]{CLS11}\footnote{In general, for a smooth projective variety $X$, a real torus $T \subset \Aut_{\rm red}(X)$ and a line bundle $L$, the bundle $L$ admits a $T$-equivariant linerization (\cf \cite[Chapter III, Proposition 9.5]{Kob95}). However, when $X$ is not projective, the existence of a $T$-equivariant linearization on $\cO_X(D)$ is obstructed even if the divisor $D \subset X$ is $T$-invariant \cite[Theorem 2]{KS24}.}. We take a defining section $s_i$ of $D_i$ and $T_Y$-invariant fiber metric $h_i$ on $\cO_Y(D_i)$. Then the curvature $\xi_i$ of $h_i$ has an expression
\begin{equation} \label{inner product of xi with v}
\xi_i=-\dd \log |s_i|_{h_i}^2
\end{equation}
on $Y \backslash D_i$. In particular we have a uniform bound
\[
-C \omega \leq \dd \log |s_i|_{h_i}^2 \leq C \omega
\]
on $Y \backslash D_i$. For any $t \in T_Y$, both $s_i$ and $t \cdot s_i$ are defining sections of $D_i$, which implies that $t \cdot s_i=\chi(t) s_i$ for some character $\chi \colon T_Y \to \C^\ast$. Since $T_Y$ is compact, the image of $\chi$ must be a compact subgroup of $\C^\ast$, which shows that $|\chi(t)|=1$ for all $t \in T_Y$. In particular, we know that $|s_i|_h^2$ is $T_Y$-invariant and $\Im(v) \log |s_i|_h^2=0$ on $Y \backslash D_i$. Applying $i_v$ to \eqref{inner product of xi with v} we get
\begin{equation}
i_v \xi_i=-\frac{\i}{2 \pi} \bp \big( \Re(v)(\log |s_i|_{h_i}^2) \big),
\end{equation}
and hence
\[
\big| \Re(v)(\log |s_i|_{h_i}^2) \big| \leq C
\]
on $Y \backslash D_i$. Thus for sufficiently small $\kappa>0$ the function
\[
\psi:=\phi+\kappa \sum_i \log |s_i|_{h_i}^2
\]
satisfies
\[
Q_{Y,\omega}\bigg( \hat{\chi}+\dd \psi \bigg)-\theta_v^X(\hat{\chi})-\Re(v)(\psi) \leq c_X-\d
\]
on $Y \backslash D$.

Now let $Y \subset X$ be a $p$-dimensional toric subvariety, $Y_k$ ($k=1,\ldots,\ell$) its $p$-dimensional irreducible components, and $Y_k' \subset Y_k$ the complement of toric divisors in $Y_k$. By the induction hypothesis applied to $Z:=Y \backslash \cup_k Y_k'$, there exists a neighborhood $U$ of $Z$ in $X$ and a $\hat{\chi}$-K\"ahler potential $\phi_Z$ such that
\[
Q_{X,\omega} \bigg( \hat{\chi}+\dd \phi_Z \bigg)-\theta_v^X(\hat{\chi})-\Re(v)(\phi_Z)<c_X
\]
on $U$. On the other hand, applying the above construction to each $Y_k$, we obtain $\psi_k$ such that
\[
Q_{Y_k,\omega}\bigg( \hat{\chi}+\dd \psi_k \bigg)-\theta_v^X(\hat{\chi})-\Re(v)(\psi_k) \leq c_X-\d
\]
on $Y_k'$ and $\psi_k \to -\infty$ along the toric divisors of $Y_k$. For simplicity, we assume that $\ell=1$ (the general case follows by applying the same argument to each $Y_k$). We can choose open neighborhoods $W \Subset V \Subset U$ of $Z$ in $X$ such that
\begin{enumerate}
\item  $\psi_1>\phi_Z+2$ in $Y_1 \cap (U \backslash V)$
\item $\psi_1<\phi_Z-2$ in $Y_1 \cap W$
\end{enumerate}
by subtracting a sufficiently large constant from $\phi_Z$, since $\psi_1 \to -\infty$ along the toric divisors of $Y_1$. We take a smooth extension of $\psi_1$ near $Y_1 \backslash W$ and still denote it by $\psi_1$. Consider
\[
\tilde{\psi}_1:=\psi_1+B d_{Y_1}^2,
\]
where $B>0$ is a constant and $d_{Y_1}$ is the distance function from $Y_1$. By taking $B$ sufficiently large, we have
\[
Q_{X,\omega} \bigg( \hat{\chi}+\dd \tilde{\psi}_1 \bigg)-\theta_v^X(\hat{\chi})-\Re(v)(\tilde{\psi}_1)<c_X
\]
in an open neighborhood $\tilde{U}$ of $Y_1 \backslash W$ in $X$, since $\Re(v)(d_{Y_1}^2)=0$ along $Y_1$. By continuity of $\psi_1$, after possibly shrinking $\tilde{U}$, we also have
\[
\tilde{\psi}_1>\phi_Z+1
\]
in $\tilde{U} \cap (U \backslash V)$ and
\[
\tilde{\psi}_1<\phi_Z-1
\]
in $\tilde{U} \cap W$. Let $\varphi$ be the regularized maximum of $(\tilde{U},\tilde{\psi}_1)$ and $(V,\phi_Z)$. From the above observations, we conclude that
\[
Q_{X,\omega} \bigg( \hat{\chi}+\dd \varphi \bigg)-\theta_v^X(\hat{\chi})-\Re(v)(\varphi)<c_X
\]
on a small neighborhood of $Y$ in $X$. This completes the proof.
\end{proof}

\begin{proof}[Proof of Theorem \ref{NM criterion for smooth projective toric varieties}]
In Step 1 of Theorem \ref{subsolutions near subvarieties}, let $Y=X$ and $D$ the union of all toric divisors of $X$. Then by Theorem \ref{subsolutions near subvarieties}, we know that $\Gamma_{\omega,\hat{\chi}}(D) \neq \emptyset$. Hence by proceeding with the same argument, we see that the continuity path
\[
\Tr_{\chi_{\phi_t}} \omega+b_t \frac{\omega^n}{\chi_{\phi_t}^n}=c_t+\theta_v^X(\chi_{\phi_t}), \quad b_t=c_X \frac{\b^n}{\a^n}t, \quad c_t=c_X(t+1)
\]
for $\chi_{\phi_t}=\hat{\chi}+\dd \phi_t \in \b$ ($\phi_t \in \cH(X,\hat{\chi})^T$) has a solution for all $t \in [0,\infty)$. This completes the proof.
\end{proof}

\subsection{Interior $C^2$-estimates} \label{Interior C2-estimates}
The goal of this subsection is to prove the interior $C^2$-estimates required in Step 1 of Theorem \ref{subsolutions near subvarieties}. Let $X$ be an $n$-dimensional smooth projective toric variety. We identity the open dense $T$-orbit in $X$ with $\R^n \times (\S^1)^n$ and express $\omega$, $\chi$ as
\[
\omega=\dd f, \quad \chi=\dd g
\]
for some smooth strictly convex function $f, g \colon \R^n \to \R$. Assume that $f$, $g$ satisfy the equation
\begin{equation} \label{generalized modified J-equation on principal T-orbit}
g^{ij}f_{ij}+b \frac{\det(D^2 f)}{\det(D^2 g)}=c+\ell(\nabla g)-\frac{1}{V} \int_P \ell dy,
\end{equation}
where $b \geq 0$ and $c>0$ are constants, $V:=\Vol(P)$, $\ell \colon \R^n \to \R$ is a linear function, and $P=\nabla g (\R^n)$ is the moment polytope determined uniquely from the K\"ahler class $\b$ up to parallel translations. Also we assume that 
\begin{equation} \label{positive lower bound delta}
c+\inf_P \ell-\frac{1}{V} \int_P \ell dy \geq \d
\end{equation}
for some constant $\d>0$. Adding an affine linear function to $g$ amounts to a parallel translation of the moment polytope $P$. Since the equation \eqref{generalized modified J-equation on principal T-orbit} and the condition \eqref{positive lower bound delta} remain unchanged under this operation, we may assume that $g(0)=0$ and $\nabla g(0)=0$. In particular, since $0 \in P$ we obtain:

\begin{lem}[\cite{CS17}, Lemma 26]
For any compact subset $K \subset \R^n$ there exists $C>0$ such that $\sup_K |g|<C$.
\end{lem}

Let $h \colon P \to \R$ be the Legendre transform of $g$. Then we observe that
\begin{equation} \label{Legendre transform}
\nabla h(y)=(\nabla g)^{-1}(y), \quad D^2 h(y)=\big( D^2 g(\nabla h(y)) \big)^{-1}.
\end{equation}
By using \eqref{Legendre transform} we can translate the equation \eqref{generalized modified J-equation on principal T-orbit} to
\begin{equation} \label{generalized modified J-equation for Legendre dual}
f_{ij}(\nabla h(y))h_{ij}+b \det(D^2 f(\nabla h(y))) \det (D^2 h(y))=c+\ell(y)-\frac{1}{V} \int_P \ell dy.
\end{equation}
Our strategy is essentially the same as \cite[Section 5.1]{CS17}, namely, we establish higher order estimates for the Legendre transform $h$. If there is a sequence of equations \eqref{generalized modified J-equation on principal T-orbit} that violates the $C^2$ bound, we take the limit $k \to \infty$ and apply Bian--Guan's constant rank theorem \cite[Theorem 1.1]{BG09} to obtain a contradiction.

\begin{prop} \label{C2 estimate for g}
Let $B \subset \R^n$ be the unit ball, and $f$, $g$ satisfy \eqref{generalized modified J-equation on principal T-orbit} with normalization $\inf_B g=g(0)=0$. Then there exists $C>0$ depending only on $\sup_B |g|$, bounds on $b$, $c$, the coefficients of $\ell$, $C^{3,\g}$ bounds on $f$ and a positive lower bound on the Hessian of $f$ such that
\[
\sup_{\frac{1}{2}B} |g_{ij}|<C.
\]
\end{prop}
\begin{proof}
Suppose that there are sequences $f_k$, $g_k$, $b_k$, $c_k$, $\ell_k$ satisfying the hypothesis, including $|g_k|<N$ for some constant $N>0$, but $|D^2 g_k(x_k)|>k$ for some $x_k \in \frac{1}{2} B$. The moment polytopes $P_k=\nabla g_k(\R^n)$ of $g_k$ satisfy $0 \in P_k$ and are parallel to each other. So by passing to a subsequence we may assume that $P_k \to P_\infty \subset \R^n$ via parallel translations, and hence $\int_{P_k} \ell_k dy \to \int_{P_\infty} \ell_\infty dy$ as $k \to \infty$. By shrinking the ball a bit we can assume that $g_k \to g$ uniformly for some strictly convex $g \colon B \to \R$. Also, by \eqref{generalized modified J-equation on principal T-orbit} we have
\[
g_k{}^{ij} f_{k,ij} \leq c_k+\max_{P_k} \ell_k-\frac{1}{V} \int_{P_k} \ell_k dy.
\]
Since the right hand side has a uniform upper bound, combining with a uniform lower bound for $D^2 f_k$ yields
\begin{equation} \label{lower bound of g}
g_{ij}>\tau \d_{ij}
\end{equation}
for some $\tau>0$ independent of $k$. Let $h_k$ be the Legendre transform of $g_k$ and set $U_k:=(\nabla g_k)(B) \subset P_k$. Then we have $h_k(0)=0$, $\nabla h_k(0)=0$ from the normalization. Moreover, the formula \eqref{Legendre transform} together with \eqref{lower bound of g} yields an upper bound of $D^2 h_k$, and hence $h_k$ is uniformly bounded in $C^2$. Let $\nabla g$ denote the subdifferential of $g$, \ie
\[
(\nabla g)(x):=\{p \in \R| \text{$g(x') \geq g(x)+p(x'-x)$ for all $x' \in B$} \}, \quad x \in B
\]
and
\[
(\nabla g)(E):=\bigcup_{x \in E} (\nabla g)(x)
\]
for any subset $E \subset B$. Then the estimate for subdifferentials \cite[Lemma 28]{CS17} implies that for sufficiently large $k$ we have $(\nabla g)(0.9B) \subset U_k$, and so $(\nabla g)(0.8B)$ remains a fixed positive distance from $\p U_k$ for large $k$. In addition $h_k$ solves the equation
\[
f_{k,ij}(\nabla h_k(y))h_{k,ij}+b_k \det(D^2 f_k (\nabla h_k(y))) \det (D^2 h_k(y))=c_k+\ell_k (y)-\frac{1}{V} \int_{P_k} \ell_k dy.
\]
We use Proposition \ref{C2 gamma estimate for h} together with the Schauder estimate, to obtain uniform $C^{3,\g}$ bounds for $h_k$ on $\nabla g(0.8B)$. So combining with bounds for $f_k$, $b_k$, $c_k$, $\ell_k$, we can extract a convergent subsequence $h_k \to h_\infty$ on $\nabla g (0.8B)$ in $C^{3,\g/2}$ satisfying an equation of the form
\[
f_{\infty,ij}(\nabla h_\infty(y))h_{\infty,ij}+b_\infty \det(D^2 f_\infty (\nabla h_\infty(y))) \det (D^2 h_\infty(y))=c_\infty+\ell_\infty(y)-\frac{1}{V} \int_{P_\infty} \ell_\infty dy.
\]
As in \cite[Proposition 27]{CS17} we consider the following two cases separately:
\begin{itemize}
\item If we have a positive lower bound on $D^2 h_\infty$ this yields a uniform lower bound on $D^2 h_k$ on $\nabla g(0.8B)$, and so a uniform upper bound for $D^2 g_k$ at all points for which $\nabla g_k(x) \in \nabla g(0.8B)$. However, we have $\nabla g_k(0.7B) \subset \nabla g (0.8B)$ for sufficiently large $k$ by \cite[Lemma 28]{CS17}, and hence an upper bound of $D^2 g_k$ on $0.7B$. This is a contradiction.

\item If the Hessian $D^2 h_\infty$ degenerates as some point then we apply the Bian--Guan's constant rank theorem \cite[Theorem 1.1]{BG09} to know that $D^2 h_\infty$ degenerates everywhere. Indeed, for a fixed value of $\nabla h_\infty$, the equation is of the form
\[
F(A,y):=\Tr(BA)+b \det(A)-L(y)=0
\]
for a positive definite Hermitian matrix $B$, constant $b \geq 0$ and affine linear function $L$. The assumptions of \cite[Theorem 1.1]{BG09} are satisfied, since $F(A^{-1},y)$ is convex in $(A,y)$ according to Proposition \ref{strong convexity for positive b}. Thus we have
\[
\int_{\nabla g(0.8B)} \det (D^2 h_{\infty})=0.
\]
On the other hand,
\[
\int_{\nabla g_k(0.7B)} \det (D^2 h_k)=\Vol(0.7B).
\]
This contradicts the fact that $\nabla g_k (0.7B) \subset \nabla g(0.8B)$ for sufficiently large $k$ and $h_k \to h_\infty$ in $C^{3,\g/2}$.
\end{itemize}
\end{proof}

Hence, the local estimate for $D^2 g$ can be reduced to the local $C^{2,\gamma}$-estimate for $h$, which we will focus in the following. A main difference from the usual $J$-equation case \cite[Proposition 29]{CS17} is that we have to deal with an affine linear function in the right hand side of \eqref{generalized modified J-equation for Legendre dual}. However we will observe that the affine linear function converges to a positive constant through the blowup argument. So the Liouville rigidity result \cite[Lemma 31]{CS17} can also be applied to our case.

\begin{prop} \label{C2 gamma estimate for h}
Suppose that $h$ is a smooth convex function on the unit ball $B \subset \R^n$ satisfying the equation
\begin{equation} \label{equation for h}
a_{ij}(\nabla h)h_{ij}+b(\nabla h) \det(D^2 h)=L,
\end{equation}
where $a_{ij}, b \in C^{1,\g}$, $b \geq 0$, $\lambda<a_{ij}<\Lambda$ for some constants $\lambda, \Lambda>0$, $L$ is an affine linear function with $\inf_B L \geq \d$ for some constant $\d>0$. Then there exists a constant $C>0$ depending on $\lambda$, $\Lambda$, $\d$, bound of the coefficients of $L$, $C^{1,\g}$ bounds of $a_{ij}$ and $b$, and the $C^2$ bounds of $h$, such that
\[
\|h\|_{C^{2,\g}(\frac{1}{2}B)}<C.
\]
\end{prop}

\begin{proof}
Let
\[
N_h:=\sup_{y \in B} d_y|D^3 h(y)|,
\]
where $d_y:=d(y,\p B)$ denotes the distance from the boundary $\p B$. Assume that $N_h>1$ and the supremum is achieved at a point $y_0 \in B$. Define the function
\[
\tilde{h}(y)=d_{y_0}^{-2}N_h^2 h(y_0+d_{y_0}N_h^{-1}y)-\tilde{L}(y),
\]
where the affine function $\tilde{L}$ is chosen so that
\begin{equation} \label{normalization for h}
\tilde{h}(0)=0, \quad \nabla \tilde{h}(0)=0.
\end{equation}
Note that the rescaled function $\tilde{h}$ is defined on the ball $B_{N_h}(0)$. Also a direct computation shows that
\[
D^2 \tilde{h}(y)=D^2h(y_0+d_{y_0}N_h^{-1}y), \quad D^3 \tilde{h}(y)=d_{y_0}N_h^{-1} D^3 h(y_0+d_{y_0}N_h^{-1}y).
\]
In particular, we have $|D^3 \tilde{h}(y)| \leq 2$ on $B_{2^{-1}N_h}(0)$. Moreover, since $|D^2 \tilde{h}|=|D^2 h|<C$ the normalization \eqref{normalization for h} implies that
\[
\|\tilde{h}\|_{C^3(B_{2^{-1}N_h}(0))}<C
\]
for a uniform constant $C>0$. One can observe that the rescaled function $\tilde{h}$ satisfies an equation of the form
\begin{equation} \label{equation for tilde h}
\tilde{a}_{ij}(\nabla \tilde{h}(y)) \tilde{h}_{ij}(y)+\tilde{b}(\nabla \tilde{h}(y)) \det(D^2 \tilde{h}(y))=L(y_0+d_{y_0}N_h^{-1}y),
\end{equation}
where the coefficients $\tilde{a}_{ij}$, $\tilde{b}$ has the same $C^0$ bounds as $a_{ij}$, $b$, whereas
\begin{equation} \label{estimate for coefficients}
\sup|\nabla \tilde{a}_{ij}| \leq d_{y_0} N_h^{-1} \sup|\nabla a_{ij}|, \quad \sup|\nabla \tilde{b}| \leq d_{y_0} N_h^{-1} \sup |\nabla b|.
\end{equation}
The differentiating the equation \eqref{equation for tilde h} and using the standard Schauder theory we know that $\nabla \tilde{h}$ is bounded in $C^{2,\g}$.

Now we argue by contradiction. So we suppose that there exists a sequence of convex functions $h_k$ on $B$ satisfying \eqref{equation for h} with coefficients $a_{k,ij}$, $b_k$, and $L_k$, and associated constant $N_{h_k}>4k$. For each $k$ we take a point $y_{0,k} \in B$ that attains the supremum of $N_{h_k}$. Then the rescaled function $\tilde{h}_k$ defined on $B_{4k}(0)$ satisfies an equation of the form
\[
\tilde{a}_{k,ij}(\nabla \tilde{h}_k(y)) \tilde{h}_{k,ij}(y)+\tilde{b_k}(\nabla \tilde{h}_k(y)) \det(D^2 \tilde{h}_k(y))=L_k(y_{0,k}+d_{y_{0,k}}N_{h_k}^{-1}y),
\]
has uniform $C^{3,\g}$ bounds on $B_k(0)$ and satisfies $|D^3 \tilde{h}_k(0)|=1$. By passing to a subsequence we may assume that $y_{0,k} \to y_{0,\infty} \in B$. Combining with bounds for $a_{k,ij}$, $b_k$, $L_k$ as well as a standard diagonal argument shows that by taking a subsequence we can extract a convex limit $\tilde{h}_\infty \colon \R^n \to \R$ in $C_{\loc}^{3,\g/2}$, satisfying $|D^3 \tilde{h}_\infty(0)|=1$ and
\[
\tilde{a}_{\infty,ij} \tilde{h}_{\infty,ij}(y)+\tilde{b}_\infty \det(D^2 \tilde{h}_\infty(y))=L_\infty (y_{0,\infty})
\]
with constant coefficients $\tilde{a}_{\infty,ij}$, $\tilde{b}_\infty$ due to \eqref{estimate for coefficients}. Since $L_\infty (y_{0,\infty}) \geq \d>0$, after a linear change of coordinates, we can apply the Liouville rigidity result \cite[Lemma 31]{CS17} to conclude that $\tilde{h}_\infty$ is a quadratic polynomial. This contradicts the fact that $|D^3 \tilde{h}_\infty(0)|=1$, completing the proof.
\end{proof}

%==============Section 6==========================
\section{A numerical criterion for extremal K\"ahler metrics} \label{A numerical criterion for extremal Kahler metrics}
Let $X$ be an $n$-dimensional compact K\"ahler manifold and $T \subset \Aut_{\rm red}(X)$ a maximal compact torus. Let $\hat{\chi}$ be a $T$-invariant K\"ahler form, and let $\b$ denote its K\"ahler class. We choose the extremal vector field $v_{\rm ext}$ so that $\Im(v_{\rm ext}) \in \ft$. The following discussion is entirely based on \cite{DJ25}. However, since there are differences in conventions such as the definitions of functionals, we briefly summarize and unify them in our setting for the reader’s convenience. We define the modified $K$-energy
\[
K_{v_{\rm ext}}(\phi)=-\int_0^1 \int_X \phi \big(S(\chi_{t \phi})-\overline{S}-\theta_{v_{\rm ext}}^X(\chi_{t \phi}) \big) \chi_{t \phi}^n dt.
\]
where $S(\chi_{t \phi})$ denotes the scalar curvature of $\chi_{t \phi}$, and $\overline{S}=\frac{n c_1(X) \cdot \b^{n-1}}{\b^n}$ its average. A critical point of the functional $K_{v_{\rm ext}}$, \ie a solution $\phi \in \cH(X,\hat{\chi})^T$ to the equation
\[
S(\chi_\phi)-\overline{S}=\theta_{v_{\rm ext}}^X(\chi_\phi)
\]
is called an extremal K\"ahler metric (\cf \cite{Cal82}). By the definition of $v_{\rm ext}$ the modified $K$-energy $K_{v_{\rm ext}}$ is invariant under the action of the complexified torus $T^\C$. Let $v$ be a holomorphic vector field on $X$ such that $\Im(v) \in \ft$, and let $\rho$ be a real closed $(1,1)$-form on $X$, which is not necessarily K\"ahler. For later use, it is convenient to extend \eqref{Jvomegab functional} to $\rho$ as well:
\[
J_v^\rho(\phi)=\int_0^1 \int_X \phi \big( \Tr_{\chi_{t \phi}} \rho-c_{[\rho]}-\theta_v^X(\chi_{t \phi}) \big) \chi_{t \phi}^n dt, \quad c_{[\rho]}:=\frac{n [\rho] \cdot \b^{n-1}}{\b^n}.
\]
It is easy to verify that $J_{s_1v_1+s_2v_2}^{s_1\rho_1+s_2\rho_2}=s_1J_{v_1}^{\rho_1}+s_2J_{v_2}^{\rho_2}$ for all $s_1,s_2 \in \R$.
We set
\[
J_{T^\C}^{\hat{\chi}}(\phi):=\inf_{\tau \in T^\C} J^{\hat{\chi}}(\phi_\tau),
\]
where the potential $\phi_\tau \in \cH(X,\hat{\chi})^T$ satisfies $\tau^\ast \chi_\phi=\chi_{\phi_\tau}$. By Remark \ref{equivalent definitions of coercivity}, the functional $J^{\hat{\chi}}$ is comparable with $I$, $J$ and hence $d_1$-distance as well. In particular, for a fixed $\phi \in \cH(X,\hat{\chi})^T$, there exists $\tau \in T^\C$ that attains the infimum in the definition of $J_{T^\C}^{\hat{\chi}}$ \cite[Theorem 3.7]{DJ25}. We say that the modified $K$-energy $K_{v_{\rm ext}}$ is $T^\C$-coercive if there exist constants $C_1, C_2>0$ such that
\[
K_{v_{\rm ext}}(\phi) \geq C_1 J_{T^\C}^{\hat{\chi}}(\phi)-C_2
\]
for all $\phi \in \cH(X,\hat{\chi})^T$. We define the relative entropy $H$ by
\[
H(\phi):=\int_X \log \frac{\chi_\phi^n}{\hat{\chi}^n} \chi_\phi^n.
\]
Using these functionals, the modified $K$-energy $K_{v_{\rm ext}}$ can be expressed as
\begin{equation} \label{Chen-Tian formula}
K_{v_{\rm ext}}(\phi)=H(\phi)+J_{-v_{\rm ext}}^{-\Ric(\hat{\chi})}(\phi).
\end{equation}
This is well-known as the Chen--Tian formula \cite{Che00,Tia97a,Tia97b}. Define the $T^\C$-reduced delta invariant of the K\"ahler class $\b$ by
\begin{equation} \label{TC-reduced delta invariant}
\d_\b^{T^\C}:=\sup \bigg\{ \d \in \R \bigg| \inf_{\phi \in \cH(X,\hat{\chi})^T} \big( H(\phi)-\d J_{T^\C}^{\hat{\chi}}(\phi) \big) \in \R \bigg\}.
\end{equation}
Note that the value $\d_\b^{T^\C}$ does not depend on the choice of reference form $\hat{\chi} \in \b$, as can be verified by an argument similar to the proof of Theorem \ref{solvability is independent of omega}. By using Tian's alpha invariant, one can also see that $\d_\b^{T^\C}>0$ \cite[Section 3.2]{DJ25}. In view of \eqref{Chen-Tian formula}, the invariant $\d_\b^{T^\C}$ quantitatively measures the coercivity of the modified $K$-energy $K_{v_{\rm ext}}$.

As an application to the existence problem of extremal K\"ahler metrics, we now prove the following theorem:

\begin{thm} \label{numerical criterion for extremal Kahler metrics on smooth projective toric varieties}
Let $X$ be an $n$-dimensional compact K\"ahler manifold, and $T$, $v_{\rm ext}$, $\hat{\chi}$, $\b$ as above. Assume that Conjecture \ref{NM criterion} holds, and that there exists a constant $\e>0$ satisfying the following properties:
\begin{enumerate}
\item The class $(\d_\b^{T^\C}-\e) \b-c_1(X)$ is K\"ahler.
\item For any $T$-invariant K\"ahler forms $\chi \in \b$,
\[
\max_X \theta_{v_{\rm ext}}^X(\chi)<n(\d_\b^{T^\C}-\e)-\frac{n c_1(X) \cdot \b^{n-1}}{\b^n}.
\]
\item For any $T$-invariant $p$-dimensional irreducible subvarieties $Y \subset X$ with $p=1,\ldots,n-1$ and any $T$-invariant K\"ahler forms $\chi \in \b$,
\[
\int_Y \bigg[ \bigg( (n-p)(\d_\b^{T^\C}-\e)-\frac{n c_1(X) \cdot \b^{n-1}}{\b^n}-\theta_{v_{\rm ext}}^X(\chi) \bigg) \chi^p+p \chi^{p-1} \cdot c_1(X) \bigg]>0.
\]
\end{enumerate}
Then the K\"ahler class $\b$ admits a $T$-invariant extremal K\"ahler metric.
\end{thm}

\begin{proof}
For each $\phi \in \cH(X,\hat{\chi})^T$, let us take $\tau \in T^\C$ that minimizes $J^{\hat{\chi}}(\phi_\tau)$ so that $J_{T^\C}^{\hat{\chi}}(\phi_\tau)=J_{T^\C}^{\hat{\chi}}(\phi)=J^{\hat{\chi}}(\phi_\tau)$. As in \cite[Lemma 5.4]{DJ25}, using the $T^\C$-invariance of $K_{v_{\rm ext}}$, Chen--Tian formula \eqref{Chen-Tian formula}, and definition of $\d_\b^{T^\C}$, we obtain
\[
\begin{aligned}
K_{v_{\rm ext}}(\phi) &=K_{v_{\rm ext}}(\phi_\tau)\\
&=H(\phi_\tau)+J_{-v_{\rm ext}}^{-\Ric(\hat{\chi})}(\phi_\tau)\\
&\geq \bigg(\frac{\e}{2}+\d_\b^{T^\C}-\e \bigg) J^{\hat{\chi}}(\phi_\tau)+J_{-v_{\rm ext}}^{-\Ric(\hat{\chi})}(\phi_\tau)-C\\
&=\frac{\e}{2} J^{\hat{\chi}}(\phi_\tau)+J^{(\d_\b^{T^\C}-\e)\hat{\chi}}(\phi_\tau)+J_{-v_{\rm ext}}^{-\Ric(\hat{\chi})}(\phi_\tau)-C\\
&=\frac{\e}{2} J_{T^\C}^{\hat{\chi}}(\phi)+J_{-v_{\rm ext}}^{(\d_\b^{T^\C}-\e)\hat{\chi}-\Ric(\hat{\chi})}(\phi_\tau)-C.
\end{aligned}
\]
Thus, showing that $K_{v_{\rm ext}}$ is $T^\C$-coercive reduces to the problem of establishing a lower bound for the functional $J_{-v_{\rm ext}}^{(\d_\b^{T^\C}-\e)\hat{\chi}-\Ric(\hat{\chi})}$. Since the class $\a:=(\d_\b^{T^\C}-\e) \b-c_1(X)$ is K\"ahler, we can choose a K\"ahler form $\omega \in \a$, and define the modified $J$-functional
\begin{equation} \label{Jvext functional}
J_{-v_{\rm ext}}^\omega (\phi)=\int_0^1 \int_X \phi \big( \Tr_{\chi_{t \phi}} \omega-n(\d_\b^{T^\C}-\e)+\overline{S}-\theta_{-v_{\rm ext}}^X(\chi_{t \phi}) \big) \chi_{t \phi}^n dt.
\end{equation}
The forms $(\d_\b^{T^\C}-\e)\hat{\chi}-\Ric(\hat{\chi})$ and $\omega$ both belong to the same class $\a$. So by \cite[Lemma 3.2]{DJ25}, or by the same argument as in the proof of Theorem \ref{solvability is independent of omega}, one can see that the difference between $J_{-v_{\rm ext}}^{(\d_\b^{T^\C}-\e)\hat{\chi}-\Ric(\hat{\chi})}$ and $J_{-v_{\rm ext}}^\omega$ is bounded by a uniform constant. Since the functional $J_{-v_{\rm ext}}^\omega$ is convex (Proposition \ref{strict convexity of the modified J-functional}), the $T^\C$-coercivity of $K_{v_{\rm ext}}$ reduces to finding a critical point of the functional \eqref{Jvext functional}, \ie a solution to the modified $J$-equation for $\phi \in \cH(X,\hat{\chi})^T$:
\begin{equation} \label{modified J-equation for extremal Kahler metrics}
\Tr_{\chi_\phi} \omega=n(\d_\b^{T^\C}-\e)-\overline{S}+\theta_{-v_{\rm ext}}^X(\chi_\phi).
\end{equation}
Therefore, assuming that Conjecture \ref{NM criterion} is true, we can apply it to the equation \eqref{modified J-equation for extremal Kahler metrics}, and deduce that $K_{v_{\rm ext}}$ is $T^\C$-coercive under the assumptions in Theorem \ref{numerical criterion for extremal Kahler metrics on smooth projective toric varieties}. Then it follows from \cite[Theorem 1.2]{NJL25} that the K\"ahler class $\b$ admits a $T$-invariant extremal K\"ahler metric. This completes the proof of Theorem \ref{numerical criterion for extremal Kahler metrics on smooth projective toric varieties}.
\end{proof}

We conclude this section with a few additional remarks. First, a variant of Theorem \ref{numerical criterion for extremal Kahler metrics on smooth projective toric varieties} was already obtained in \cite[Corollary 1.5]{Che21}. The formulation there is essentially the same as ours, except that we work in the $T^\C$-reduced setting.

We also emphasize that the above theorem is essentially established in \cite[Theorem 1.4]{DJ25}\footnote{More precisely, Delcroix--Jubert \cite{DJ25} obtain an existence result for $(\mathrm{v},\mathrm{w})$-weighted cscK metrics in their context, and extremal K\"ahler metrics arises as the special case with $\mathrm{v}=1$ and $\mathrm{w}$ being the affine linear function correspoinding to the extremal vector field $v_{\rm ext}$.}. Our contribution lies in identifying a numerical characterization of the solvability of the modified $J$-equation \eqref{modified J-equation for extremal Kahler metrics}, under the assumption that Conjecture \ref{NM criterion} holds.
In particular, the conjecture allows one to replace the analytic solvability condition by explicit intersection inequalities involving $T$-invariant subvarieties.

Moreover, in the toric case, Theorem \ref{NM criterion for smooth projective toric varieties} yields the toric analogue of Conjecture \ref{NM criterion}. In particular, a stronger statement holds: the conclusions remain valid even if $\delta_{\beta}^{T^\C}-\varepsilon$ in assumptions (1)--(3) of Theorem \ref{numerical criterion for extremal Kahler metrics on smooth projective toric varieties} is replaced by $\delta_{\beta}^{T^\C}$. This is due to the fact that, in (3), the objects under consideration are restricted to toric subvarieties, which form a finite set.
                                                                                                                                                                                                                                                                                                                                                                                                                                                                                                                                                                                                                                                                                                                                                                                                                                                                                                                                                                                                                                                                                                                                                                                                                                                                                                                                                                                                                                                                                                                                                                                                                                                                                                                                                                                                                                                                                                                                                                                                                                                                                                                                                                                                                                                                                                                                                                                                                                                                                                                                                                                                                                                                                                                                                                                                                                                                                                                                                                                                                                                                                                                                                                                                                                                                                                                                                                                                                                                                                                                                                                                                                                                                                                                                                                                                                                                                                                                                                                                                                           
%==============References==========================

\end{document}